\documentclass{article} % LAST UPDATE, ET, 10, AUG, 2015

\usepackage[mathscr]{eucal}
\usepackage[super]{nth}
\usepackage{amsmath}
\usepackage{amssymb}
\usepackage{amsthm}
\usepackage[latin1]{inputenc}
\usepackage{pxfonts}
\usepackage{txfonts}
\usepackage{graphicx}
\usepackage[usenames,dvipsnames]{pstricks}
\usepackage{epsfig}
\usepackage{pst-grad} % For gradients
\usepackage{pst-plot} % For axes
\usepackage{ifthen}
\usepackage{color}

\newcommand{\intav}[1]{\mathchoice {\mathop{\vrule width 6pt height 3 pt depth  -2.5pt
\kern -8pt \intop}\nolimits_{\kern -6pt#1}} {\mathop{\vrule width
5pt height 3  pt depth -2.6pt \kern -6pt \intop}\nolimits_{#1}}
{\mathop{\vrule width 5pt height 3 pt depth -2.6pt \kern -6pt
\intop}\nolimits_{#1}} {\mathop{\vrule width 5pt height 3 pt depth
-2.6pt \kern -6pt \intop}\nolimits_{#1}}}

 \newcommand{\Rr}{\mathbb R}

\newcommand{\rec}{\operatorname{rec}}
\newcommand{\epi}{\operatorname{epi}}

\newtheorem{teo}{Theorem}[section]
\newtheorem{example}{Example}
\newtheorem{Definition}{Definition}[section]
\newtheorem{Lemma}{Lemma}[section]

\newtheorem{Proposition}{Proposition}[section]
\newtheorem{Remark}{Remark}[section]
\newtheorem{Assumption}{A}

\begin{document}

\title{Sharp hessian integrability estimates for nonlinear elliptic equations: an asymptotic approach}
\author{Edgard A. Pimentel
\,\,and \,\,
Eduardo V. Teixeira
}

\date{\today} %%

\maketitle

\begin{abstract}

\noindent We establish sharp $W^{2,p}$ regularity estimates for viscosity solutions of fully nonlinear elliptic equations under minimal, asymptotic assumptions on the governing operator $F$. By means of geometric tangential methods, we show that if the {\it recession} of the operator $F$ -- formally given by $F^*(M):=\infty^{-1} F(\infty M)$ --  is convex, then any viscosity solution to the original equation $F(D^2u) = f(x)$ is locally of class $W^{2,p}$, provided $f\in L^p$, $p>d$, with appropriate universal estimates. Our result extends to operators with variable coefficients and in this setting they are new even under convexity of the frozen coefficient operator, $M\mapsto F(x_0, M)$, as oscillation is measured only at the recession level. The methods further yield BMO regularity of the hessian, provided the source lies in that space. As a final application, we establish the density of $W^{2,p}$ solutions within the class of all continuous viscosity solutions, for generic fully nonlinear operators $F$. This result gives an alternative tool for treating common issues often faced in the theory of viscosity solutions.
\medskip


\medskip

\end{abstract}

%\thanks{
%E. Pimentel is financed by CNPq-Brazil.}

\section{Introduction}

In this article, we investigate interior $W^{2,p}$-regularity estimates for viscosity solutions of second order fully nonlinear elliptic equations
\begin{equation}\label{eq1}
	F(D^2u) \,= \,f(x).
\end{equation}
The key novelty of the present work is that $F\colon \mathcal{S}(d) \to \mathbb{R}$ is assumed to be convex (or concave) only at the ends of $\mathcal{S}(d)$, that is, only when $\|D^2u\| \approx \infty$. Under such a mild condition, we establish local $W^{2,p}$ estimates for solutions to \eqref{eq1} in terms of the $L^p$ norm of $f$. Important consequences for the general theory of fully nonlinear PDEs are then derived.

Regularity theory for viscosity solutions of second order fully nonlinear elliptic equations has been a central topic in the field of analysis of PDEs since the trailblazing works of N. Krylov and M. Safonov \cite{KS79,KS80} on Harnack inequality for non-divergence form (linear) elliptic equations, back at the beginning of the 1980's. In turn, solutions to homogeneous equations,
$F(D^2u) = 0,$ are of class $C^{1,\alpha}$, for some $0< \alpha< 1$. The next major chapter in the regularity theory of such equations comes a few years later when L. Evans \cite{Evans82} and N. Krylov \cite{Krylov82,Krylov83} proved separately  that solutions to convex (or concave) homogeneous equations are locally of class $\mathcal{C}^{2,\alpha}$.  Still under the (natural) assumption that $F$ is convex, L. Caffarelli, in his seminal paper \cite{MR1005611}, established the foundations of the $W^{2,p}$ theory for fully nonlinear elliptic equations.  

Whether $W^{2,p}$ regularity estimates were available for {\it any} uniformly elliptic fully nonlinear equation challenged the community for over thirty years. The problem was settled in the negative by N. Nadirashvili and S. Vl{\u{a}}du{\c{t}, see \cite{NV07,NV08, NV11}.  We also mention here the examples from \cite{MR3266252} of linear elliptic operators with piecewise constant coefficients whose solutions fail to be in $W^{2,p}$.

Sharp hessian integrability theory for Equation \eqref{eq1} is indeed an intricate mathematical puzzle. In turn, special hidden structures hold the key to the validity of $W^{2,p}$ {\it a priori} estimates. Often such hidden structures are not perceived by classical methods and techniques used in the study of PDEs, and alternative approaches must be considered. In this present work, we tackle this issue by means of the so-called geometric tangential analysis. This comprises a series of techniques relating a given problem to an auxiliary one through a genuinely geometric structure. The core of the geometric tangential analysis is to build a path that touches the original problem of interest and connects it with an auxiliary, model-problem. Among the diverse manners we have to build such a path, we bring up in this work the notion of recession function: given a fully nonlinear elliptic operator $F$ defined on $\mathcal{S}(d)$, we denote by $F^*$ the {\it limiting} operator
\begin{equation}\label{recession}
	F^*(M):=\lim_{\mu\to 0}\mu F(\mu^{-1}M),
\end{equation}
for $M\in\mathcal{S}(d)$. Such a limiting operator appears naturally in the study of free boundary problems ruled by fully nonlinear equations, where hessian blow-up is expected through the phase transition. Hence, the free boundary condition, that is, the equation satisfied along the free boundary is prescribed by the $F^{*}$ rather than $F$, see for instance \cite{RT, MR3067831}. In a related reasoning, {\it improved} $C^{1,\alpha}$ estimates based on limiting profiles of the operator have been recently announced in \cite{SilvTeix}.

Hereafter, the operator $F^*$ will be called the {\it recession} function associated with $F$.  This nomenclature is borrowed from the realm of convex analysis \cite{Rockafellar}, and can be found in applications of its techniques to numerous problems, ranging from Economics (e.g. production theory, see \cite{econ}) to Physics (e.g. continuum mechanics, see \cite{phys}). In this tradition, a convex set $A\subset\mathbb{R}^d$ is said to recedes in the direction of $y\in\mathbb{R}^d$, for $y\neq 0$, if 
\begin{equation}\label{conecond}
	x\,+\,\mu y\,\in\,A	
\end{equation}
for every $\mu\geq 0 $ and $x\in A$. The set of all vectors $y\in\mathbb{R}^d$ for which \eqref{conecond} is satisfied is called the \textit{recession cone} of $A$, denoted $\rec(A)$. Consider further a convex function $f:\Rr^d\to\Rr$ and denote its epigraph by $\epi(f)$. A straightforward reasoning yields that  $\rec(\epi(f))$ is on its turn the epigraph of a function, say $f^*$. Hence, $f^*$ is said to be the \textit{recession function} associated with $f$. It can be shown that, under a few conditions, we have
\begin{equation}\label{reccon}
	f^*(y)\,=\,\lim_{\mu\to 0}\mu\left(f\left(x+\mu^{-1} y\right)-f(x)\right);
\end{equation}
see \cite{Rockafellar}. By taking $x\equiv 0$ and assuming that $f(0)=0$, we recover \eqref{recession} through the limit in \eqref{reccon}.

Insofar as the heuristics of the geometric tangential methods are concerned, by making appropriate assumptions on $F^*$, we expect to import regularity from the ends of $\mathcal{S}(d)$ back to $F$. The tangential path in this case is parametrized by $\mu>0$. Our main result is stated in the following theorem:

%The recession function $F^*$ enables us to perform a fine analysis of the structures determining the regularity of the solutions to the original operator. Hence, it has played a significant role in establishing a variety of results on the regularity theory of PDEs. In \cite{MR3067831} the authors investigate fully nonlinear elliptic equations with singular absorption terms, and the associated free boundary problem. They establish sharp regularity properties of the solutions as well as geometric-measure properties of the free boundary. Moreover, the idea of recession function also appears in \cite{SilvTeix}, where the authors establish a priori local $\mathcal{C}^{1,Log-Lip}$ estimates for solutions of \eqref{eq1}, by assuming that $F^*$ is convex/concave.

%In the present paper we investigate a priori $W^{2,p}$ regularity for the solutions of \eqref{eq1} by assuming that the associated recession function $F^*$ has $\mathcal{C}^{1,1}$ estimates. Our main result is stated in the following theorem:

\begin{teo}[A priori $W^{2,p}$ estimate]\label{w2ptheorem}
Let $F\colon \mathcal{S}(d) \to \mathbb{R}$ be uniformly elliptic, $p>d$, and $u$ be a viscosity solution of 
$$
	F(D^2u)\,=\,f(x), \quad \text{in } B_1.
$$ 
Assume $F^{*}$ has {\it a priori} $C^{1,1}$ estimates. Then $u\in W^{2,p}(B_{1/2})$ and there exists $C  > 0$ so that
\begin{equation}\label{w2p est}
	\left\|u\right\|_{W^{2,p}(B_{1/2})}\,\leq\,C\left(\left\|u\right\|_{L^{\infty}(B_{1})}\,+\,\left\|f\right\|_{L^{p}(B_{1})}\right).
\end{equation}
\end{teo}

%The proof of Theorem \ref{w2ptheorem}, as well as the main assumptions under which we work, is detailed in Section \ref{main}.
 
Theorem \ref{w2ptheorem} accommodates a fairly general class of fully nonlinear operators $F$, as we shall exemplify when time comes. In turn, a large class of problems can be treated by the methods developed in this article. As a striking application of the results proven in this paper, we will verify that the set  of $W^{2,p}$ viscosity solutions is dense in the set of all continuous viscosity solutions of a given class of fully nonlinear equations, in a sense that will be made precise later. This fact enables one to bypass, in many cases, the formalism of the viscosity solution language, when addressing a priori regularity estimates, or any target property that is closed under uniform convergence. 

%Of particular interest, Theorem  \ref{w2ptheorem} implies, via embeddings, the main Theorem from \cite{SilvTeix}; see also the result proven in Section \ref{sct BMO}. 

%In addition, Theorem \ref{w2ptheorem} builds upon an iterative scheme introduced in \cite{MR3158810} to yield $p$-BMO estimates for the solutions of \eqref{eq1}, provided $f\in L^\infty$. This fact closely relates to the results presented in \cite{MR1978880}. %It indicates that the delicate analysis of these tangential methods reaches deeper layers of the regularity theory as well. 

The remainder of this paper is organized as follows. In section \ref{main}, we present the set-up under which we shall work in this paper. We also discuss the main ideas and insights concerning the proof of Theorem \ref{w2ptheorem}. In section \ref{sct example} we present few examples for which our results can be directly applied to. In section \ref{sct tools} we revisit  some tools and elements required in the study of hessian integrability estimates for solutions to non-divergence form equations. In section \ref{gtm} we discuss the idea of linking the regularity theory of the original operator $F$ to the (better) one of the recession operator $F^{*}$. This is done by an appropriate path along the set of all $(\lambda, \Lambda)$-elliptic equations. In section \ref{sct proof main thm} we deliver the proof of Theorem \ref{w2ptheorem} while in section \ref{sct var coeff} we discuss generalization to variable coefficient equations, $F(x, D^2u) = f(x)$. For the latter, we only measure the oscillation of the coefficients at the recession level, which may be strictly less than the oscillation of the coefficients of the original operator. Hence, Theorem \ref{w2,p for var coeff} gives a new information even in the classical setting where $M\mapsto F(x_0, M)$ is assumed to be convex. In section \ref{sct BMO} we address the borderline case $p=\infty$, that is, we provide BMO interior estimates for $D^2u$ in terms of BMO norm of the source $f$. Finally, in section \ref{sct density} we show our sharp integrability estimates can be applied to establish the density of $W^{2,p}$ solutions within the set of all $C^{0}$-viscosity solutions. 

\medskip
\noindent {\bf Acknowledgements}. This work was conducted during the authors' long term visit to the ICMC--Instituto de Ci\^encias Matem\'aticas e de Computa\c{c}\~ao, from Universidade de S\~ao Paulo in S\~ao Carlos. The authors thank the warm hospitality of that institute. EP has been partially supported by CNPq -- Brazil; ET thanks support from CNPq -- Brazil and FAPESP.

\section{Main assumptions and outline of the proof}\label{main}

In this section, we detail the main assumptions and set-up under which we shall work on the present paper. The space of all real, $d\times d$ symmetric matrices is denoted by $\mathcal{S}(d)$. This is a $\frac{d(d+1)}{2}-$dimensional space. It is a classical result that all matrices $M \in \mathcal{S}(d)$ are diagonalizable. For $M\in \mathcal{S}(d)$, we define 
$$
	\|M\| := \sum_{i=1}^d |e_i|,
$$
where $e_1, e_2, \cdots, e_d$ are the eigenvalues of $M$. We recall that any two norms in $\mathcal{S}(d)$ are equivalent. Hessians of $d$-dimensional $C^2$ functions, $D^2u$, lie in $\mathcal{S}(d)$ and in this article we are interested in second order nonlinear partial differential equations of the form $F(D^2u) = f(x),$ where $F\colon \mathcal{S}(d) \to \mathbb{R}$. Here is the main assumption upon the operator $F$, supposed to hold throughout the entire paper:

\begin{Assumption}\label{a1}
The operator $F$ is uniformly elliptic, with ellipticity constants $\lambda$, $\Lambda$. That is, for $M\in\mathcal{S}(d)$, we have
\[
	\lambda\left\|N\right\|\,\leq\,F(M+N)\,-\,F(M)\,\leq\,\Lambda\left\|N\right\|,
\]
for every $N\,\geq\,0$. We further assume, with no loss of generality, that $F(0) = 0$.
\end{Assumption}

An operator satisfying A\ref{a1} is also referred to as $(\lambda,\Lambda)$-elliptic. It is sometimes convenient to express ellipticity  in terms of the extremal Pucci operators:
$$
	\begin{array}{lll}
		\mathscr{P}^{-}_{\lambda, \Lambda}(M) &:=& \Lambda \cdot \text{Trace}(M^{-}) + \lambda \cdot \text{Trace}(M^{+}), \\
		\mathscr{P}^{+}_{\lambda, \Lambda}(M) &:=& \Lambda \cdot \text{Trace}(M^{+}) + \lambda \cdot \text{Trace}(M^{-}).
	\end{array}
$$
Assumption A\ref{a1} is equivalent to
$$
	\mathscr{P}^{-}_{\lambda, \Lambda}(M-N) \le F(M) - F(N) \le \mathscr{P}^{+}_{\lambda, \Lambda}(M-N),
$$
for all $M, \, N \in \mathcal{S}(d)$.  Under such a monotonicity assumption on $F$, the notion of viscosity solutions provides an appropriate definition of weak solutions for Equation \eqref{eq1}.

\begin{Definition}A continuous function $u \in C^0(B_1)$ is said to be a viscosity subsolution to \eqref{eq1} in $B_1$ if whenever one touches the graph of $u$ from above by a smooth function $\varphi$ at $x_0 \in B_1$ (i.e. $\varphi - u$ has a local minimum at $x_0$), there holds
$$
    F(D^2\varphi(x_0)) \ge f(x_0).
$$
Similarly, $u$ is a viscosity supersolution to \eqref{eq1} if whenever one touches the graph of $u$ from below by a smooth function $\phi$ at $y_0 \in B_1$, there holds
$$
    F(D^2\phi(y_0)) \le f(y_0).
$$
We say $u$ is a viscosity solution to \eqref{eq1} if it is a subsolution and a supersolution.
\end{Definition} 

The key goal of this present work is to show that sharp hessian integrability estimates for viscosity solutions to \eqref{eq1} is  controlled by the behavior of the operator $F$ at the ends of its space of definition. This brings us to the notion of recession function, which we formally define in the sequel.
\begin{Definition}[Recession function]\label{recession1}
The recession function $F^*(M)$ associated with the fully nonlinear operator $F$ is given by
$$
	F^*(M) := \lim_{\mu\to 0}F_\mu(M) := \lim_{\mu\to 0}\mu F(\mu^{-1}M).
$$ 
\end{Definition}

Notice that $F$, $F_\mu:= \mu F(\mu^{-1}M)$ and $F^*$ have the same ellipticity constants. It is also relevant to observe that the recession function is always homogeneous of degree one. The primary assumption we make in this article as to derive sharp $W^{2,p}$ estimates for viscosity solutions of \eqref{eq1} is that the recession operator $F^{*}$ has a rich enough {\it a priori} regularity theory.

\begin{Assumption}\label{a2}
We assume that the recession function $F^*$ associated with the operator $F$  exists  and has a priori $\mathcal{C}^{1,1}_{loc}$ estimates. That means $F^{*}(D^2h) = 0 $ in $B_1$ in the viscosity sense, implies $h\in C^{1,1}(B_{1/2})$ and
$$
	\|h\|_{C^{1,1}(B_{1/2}} \le C_* \|h\|_{L^\infty(B_1)},
$$
for a constant $C_{*} \ge 1$.
\end{Assumption}

In particular, if $F^*$ happens to be concave (or convex), A\ref{a2} is immediately satisfied due to Evans-Krylov Theorem.  This is the case when $F$ is suitably modified outside a ball of $\mathcal{S}(d)$. For instance, if $F$ is supposed to be concave (or convex) in $\mathcal{S}(d)\setminus B_R$, for some $R\gg 0$. This constitutes an important class of examples and we write it down for future references.

\begin{example}\label{example0} Let $F, G \colon \mathcal{S}(d) \to \mathbb{R}$ be uniformly elliptic operators with, say, $G$ homogeneous of degree one and convex. Assume $F(M) = G(M)$ for all $\|M\| \ge R$, for some $R\gg 1$. Then $F^{*}(M) = G (M)$.
\end{example}

We further comment that existence and uniqueness of the recession function is not {\it per se} required. By ellipticity, $\{F_\mu\}_{\{\mu>0\}}$ is locally pre-compact in $\mathcal{S}(d)$, hence, up to a subsequence, $F_\mu$ always converge to a recession function $F^{*}$. Assumption A\ref{a2} should be understood as a condition on {\it any} recession function of $F$.

Per the classical constraints on the notion of viscosity solutions, we will require that the source function $f$ is continuous in $B_1$.
\begin{Assumption}\label{a3}
We assume that $f\in C^0(B_1)$.
\end{Assumption}

Upon such a constraint, and in accordance to Caffarelli's theory, our $W^{2,p}$-regularity estimate is understood as an {\it a priori} estimate. By using weaker notions of viscosity solutions, it is possible to work under integrability assumption on $f$, c.f. \cite{CCKS}.  This is an immediate consequence of the stability properties of $L^p$ viscosity solutions. Besides, we expect that one can relax the integrability condition on $f$ in the sense of Escauriaza, see \cite{MR1237053}. 

\subsection{Outline of the proof}\label{main2}

Through the rest of this section, we will discuss the insights and main strategies for proving Theorem \ref{w2ptheorem}. Intuitively, $F^{*}$ ``governs" Equation \eqref{eq1} in the region where the hessian of $u$ blows-up. Hence, a priori $C^{1,1}$ estimates available for $F^{*}$ set a competing inequality which, in turn, should yield a better measure decay on 
$$
	\Theta_K := \ \{ x\text{ in } B_{1/2} : D^2u(x) > K  \},
$$
for  $K\gg 1$ sufficiently large. Of course, there are obvious difficulties in carrying out the above mentioned reasoning. For instance, in principle there is no information on the set where the hessian will be large, and in general  $\Theta_K$ are very irregular sets. We should also caution the readers to the existence of a $C^{1,1} \setminus C^2$ viscosity solution to a fully nonlinear elliptic equation, \cite{NV07}. Hence, it is not possible to establish continuity of the hessian out from the regularity theory of its recession function, even if $F^{*} = \Delta$. 

Alternatively, we will approach the problem with the aid of geometric tangential methods, where  $F^{*}$, along with its regularity theory, is regarded as a {\it target} profile. If we denote by $F_\mu := \mu F(\mu^{-1}M)$, the map $\mu \mapsto F_\mu$ provides a {\it tangential} path within the manifold consisting of $(\lambda, \Lambda)$-elliptic equations linking the regularity theory of $F_1 = F$ with the (better) one available for $F^{*}$.  
\begin{figure}[h!]
\scalebox{0.9} % Change this value to rescale the drawing.
{
\begin{pspicture}(0,-4.6849804)(13.674902,4.6849804)
\definecolor{color3217b}{rgb}{0.8,0.8,0.8}
\psbezier[linewidth=0.08](0.7014301,-0.7923486)(0.63701165,-0.66540885)(1.5752097,4.5550194)(6.9920535,4.449756)(12.408897,4.3444934)(12.877011,0.16512759)(12.838286,0.1971248)(12.7995615,0.229122)(10.097457,1.7971247)(8.97656,-1.4239278)(7.855662,-4.6449804)(8.547479,-0.30813837)(5.135265,0.36554584)(1.723051,1.0392301)(0.76584846,-0.9192884)(0.7014301,-0.7923486)
\psbezier[linewidth=0.02](0.73701173,-0.80498046)(1.8370117,-1.6449804)(5.777012,-1.6449804)(6.9370117,-0.6649805)
\psdots[dotsize=0.12](2.4370117,1.1150196)
\usefont{T1}{ptm}{m}{n}
\rput(2.2984667,0.88001955){$F$}
\psbezier[linewidth=0.04](2.4770117,1.1750195)(4.5460863,4.074108)(8.657012,4.6350193)(11.257011,1.6150196)
\usefont{T1}{ptm}{m}{n}
\rput(6.908467,3.1600196){$F_\mu$}
\usefont{T1}{ptm}{m}{n}
\rput(3.7892382,-2.2799804){manifold of $(\lambda, \Lambda)$-elliptic operators}
\rput{-42.0}(2.253172,8.146433){\psellipse[linewidth=0.02,dimen=middle,fillstyle=solid,fillcolor=color3217b](11.737678,1.1383597)(0.6901146,0.37888086)}
\psdots[dotsize=0.16,fillstyle=solid,dotstyle=o](11.777012,1.1350195)
\psline[linewidth=0.04cm,linestyle=dotted,dotsep=0.16cm](11.257011,1.6150196)(11.757011,1.1550195)
\usefont{T1}{ptm}{m}{n}
\rput(11.948467,0.9200195){$F^{*}$}
\psline[linewidth=0.02cm,tbarsize=0.07055555cm 5.0,arrowsize=0.05291667cm 2.0,arrowlength=1.4,arrowinset=0.4]{|->}(11.517012,1.5750195)(11.497012,3.7550194)
\usefont{T1}{ptm}{m}{n}
\rput(11.516465,4.1000195){$W^{2,p}$-regularity theory}
\psbezier[linewidth=0.002,linecolor=white,fillstyle=crosshatch*,hatchwidth=0.01,hatchangle=42.0,hatchsep=0.121199995](0.9370117,-0.8707513)(0.6970117,-0.78648007)(0.9970117,-0.3860518)(2.0570116,0.11448386)(3.1170118,0.61501956)(4.5570116,0.5267067)(5.5370116,0.19456957)(6.5170116,-0.13756756)(6.9370117,-0.6684893)(6.877012,-0.7063946)(6.817012,-0.74429995)(5.5570116,-1.4049804)(3.6770117,-1.376443)(1.7970117,-1.3479055)(1.1770117,-0.95502263)(0.9370117,-0.8707513)
\psline[linewidth=0.02cm,tbarsize=0.07055555cm 5.0,arrowsize=0.05291667cm 2.0,arrowlength=1.4,arrowinset=0.4]{|->}(5.6570115,0.57501954)(3.5570116,-2.1049805)
\usefont{T1}{ptm}{m}{n}
\rput(11.528242,4.5000196){Neighborhood with}
\psline[linewidth=0.12cm,arrowsize=0.05291667cm 2.0,arrowlength=1.4,arrowinset=0.4]{->}(6.697012,3.6150196)(6.9970117,3.6350195)
\end{pspicture} 
}
\vspace{-2cm}
\caption{Heuristics of the strategy to be implemented in the proof of Theorem \ref{w2ptheorem}. For $\mu = \mu(p)$ sufficiently small, but still positive, $F_\mu$ enters within a neighborhood of $F^{*}$ for which $W^{2,p}$ estimates are available. Such an improved estimate is transported down to the original operator $F = F_1$ through the tangential path.}
\end{figure}
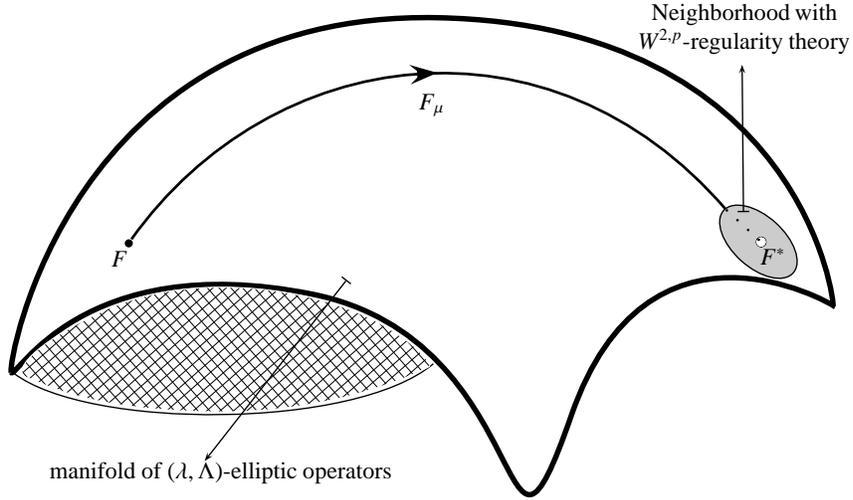

Within the heuristics of geometric tangential methods,  {\it universal} regularity theory should be understood as a relatively open property, in the sense that if $F^{*}$ has {\it a priori} $C^{1,1}_\text{loc}$ estimates then near $F^{*}$ operators should provide a ``slightly weaker" smoothness effect.  Hence, owing to $C^{1,1} \simeq W^{2,\infty}$ estimates for $F^{*}$, one should be able to reach $W^{2,p}$ regularity, $p<\infty$, for operators ``near" $F$. Hence, one should expect that for $0< \mu \le  \mu_0 \ll 1$, the operator $F_\mu$ belongs to such a neighborhood and then one may transfer such an estimate back to the original $F = F_1$.

In section \ref{gtm} we recur to compactness methods to prove an appropriate approximation lemma, which justifies rigorously the above discussion. Carrying out this analysis, with additional results and techniques from \cite{MR1005611}, we can deliver a proof of the following result:

\begin{Proposition}\label{prop1}
Let $u$ be a bounded viscosity solution of \eqref{eq1} and assume that A\ref{a1}-\ref{a3} hold. Suppose further that
$$
	\left\|u\right\|_{L^\infty(B_1)}\,\leq\,1\,\,\,\mbox{and}\,\,\,\left\|f\right\|_{L^P(B_1)}\,\leq\,\epsilon,
$$
for a small $\epsilon>0$. Then, $u\in W^{2,p}(B_{1/2})$ and 
$$
	\left\|u\right\|_{W^{2,p}(B_{1/2})}\,\leq\,C,
$$
where $C>0$ depends on the dimension $d$, $F$ and a priori $C^{1,1}$ estimates for $F^{*}$.
\end{Proposition}

Routine arguments combined with the conclusion of Proposition \ref{prop1} yield a proof of Theorem \ref{w2ptheorem}.

\begin{Remark}
The $W^{2,p}$ estimate from Theorem \ref{w2ptheorem} depends not only on universal constants, but actually on {\it a priori} $C^{1,1}$ estimates for $F^{*}$ and on the modulus of convergence $F_\mu \to F^{*}$. Namely, in Lemma \ref{luc}, to be later presented,  we can define $\omega \colon (0,1) \to \mathbb{R}_{+}$ as
$$
	\omega(\epsilon) := \sup\left \{ \mu_0 > 0 : \left | \mu F(\mu^{-1}M)\,-\,F^*(M)\right | \,\leq\,\epsilon\left(1+\left\|M\right\|\right), \forall M \in \mathcal{S}(d), \, \forall \mu \le \mu_0 \right \}.
$$
The constant $C>0$ appearing in $W^{2,p}$ estimate \eqref{w2p est} depends on dimension, $\lambda$, $\Lambda$,  $C_*$ and $\omega$. 
\end{Remark}
%%%%%%%%%%%%%

\section{Further examples}\label{sct example}

Problems in differential geometry have been a profitable source of examples of fully nonlinear partial differential equations. Let $S\subset \mathbb{R}^{d+1}$ be a hypersurface and let $A$ denote its second fundamental form. Notable examples of fully nonlinear elliptic equations arising in differential geometry are of the form
\[
	F(A)\,=\,F(\kappa_1,...,\kappa_d)\,=\,\phi,
\]
where the eigenvalues $\kappa_1,...,\kappa_d$ of $A$ are the principal curvatures of $S$, and $\phi$ is a given function. When $\phi\,>\,0$ and $F$ is symmetric, $S$ is said to be a Weingarten surface. Specializing $F$ to be an $r$-th elementary symmetric function $\sigma_r$ and setting $\phi\equiv 1$ recovers important geometric quantities. For example, up to a constant, $r=1$ gives the mean curvature, $r=2$ yields the scalar curvature, and $r=d$ gives the Gauss-Kronecker curvature of $S$. See \cite{MR1404014}

Another remarkable class of fully nonlinear equations appears in the study of certain absolutely area minimizing submanifolds of $\mathbb{C}^d\approx\Rr^{2d}$. These are referred as special Lagrangian manifolds and arise, for example, in calibrated geometry. These objects retain an intrinsic connection with fully nonlinear elliptic PDEs, through what is called the special Lagrangian equation
\begin{equation}\label{SL}
	\mathcal{L}(M) := \,\sum_{i=1}^d\arctan \lambda_i = \Theta.
\end{equation}
If  $M=(x,\nabla u(x))$ is a minimal surface in $\Rr^{2d}$, then, it is a special Lagrangian manifold and intrinsically geometric information about $M$ may depend on the regularity of the hessian of $u$. It is known that the level set of \eqref{SL} are convex if and only if $|\Theta| \ge (d-2)\pi/2$, see \cite{Yuan1}. When $d=3$ and $|\Theta| < \pi/2$, given any $\delta>0$, the authors in \cite{WY} build up solutions to \eqref{SL} that are not in $C^{1,\delta}$.

While the special Lagrangian equation is not uniformly elliptic, it behaves like so for $C^{1,1}$ solutions. Hence, it is natural to inquire what is the smoothing effect of $\mathcal{L}$ when added small increment diffusions, say $\mathcal{L}_\epsilon := \mathcal{L} + \epsilon \Delta$.

\begin{example}[Perturbation of the special Lagrangian equation] \label{example3} Let $0<\alpha_1, ~\alpha_2, \cdots, \alpha_d < +\infty$ and consider 
$$
	F(M) := \,\sum_{i=1}^d\left( \alpha_i \lambda_i + \arctan \lambda_i \right ).
$$
This is a uniformly elliptic operator, and one easily computes
$$
	F^{*}(M) =  \alpha_i \lambda_i;
$$
i.e., the recession function of the perturbed special Lagrangian operator is, up to a change of variables, the laplacian operator. From Theorem \ref{w2ptheorem}, such operator has {\it a priori} $W^{2,p}$ estimates. In view of the counterexamples from \cite{WY}, these estimates cannot be uniform with respect to $|(\alpha_1, \cdots, \alpha_d)|$. Nonetheless, it follows that any (possibly singular) solution of \eqref{SL} can be approximated by $W^{2,p}$ solutions of small perturbations of the original operator.
\end{example}
%\begin{Assumption}\label{a3}
%Fix $R>0$. We assume that $F(M)$ is concave in $M$, in  $\mathcal{S}(d)\setminus B_R$. 
%\end{Assumption}

%In brief, Example \ref{example3} indicates that information pertaining purely to the realm of Differential Geometry can be unveiled by the techniques within the geometric tangential methods - in this case, the notion of recession function.

Regularity estimates based on the analysis of recession functions turn out to be an efficient tool for analyzing perturbation of geometric equations. In the sequel, we analyze a few further examples for which our main theorem can be applied.

\begin{example}\label{example1}
Let $q\in2\mathbb{N}+1$ be an odd number. Consider the eigenvalue  ``$q$-momentum" operator
\[
	 F_q(M)\,=\,F_q(\lambda_1,\cdots,\lambda_d)\,=\,\sum_{i=1}^d\left(1+  \lambda_i^q\right)^{1/q} - d,
\]
where $\lambda_1,\cdots,\lambda_d$ are the eigenvalues of the matrix $M\in\mathcal{S}(d)$. Notice that $F_q$ is neither concave nor convex. However, one easily computes
\[
	\mu F_q\left(\mu^{-1}M\right)\,=\,\sum_{i=1}^d\left(\mu^q+  \lambda_i^q\right)^{1/q} - \mu d.
\]
Thus,
\[
	F_q^*(M)\,=\,\lim_{\mu\to 0}\mu F_q\left(\mu^{-1}M\right)\,=\,\sum_{i=1}^d \lambda_i,
\]
i.e., $F_q^*$ is the laplacian operator,  for any $q\in2\mathbb{N}+1$. One easily observes that the recession convergence above is uniform in $q$; not ellipticity though. The limiting operator obtained as $q\to \infty$,
$$
	F_\infty(M) := \left \{ 
				\begin{array}{cll}
					0 &\text{for}& \|M\| <  1\\
					\text{Trace}(M) & \text{for}& \|M\| \ge 1,
				\end{array} 
			\right.
$$
prescribes no equation within the unit ball $B_1$. It would be interesting, in future research, to investigate whether $W^{2,p}$-regularity estimate provided by Theorem \ref{w2ptheorem} is uniform in $q$. 
\end{example}

Another illustratative example we bring up here concerns equations with trigonometric oscillatory dependence:

\begin{example}  \label{example4} Let $0<\alpha_1, ~\alpha_2, \cdots, \alpha_d < +\infty$ and consider 
$$
	F(M) := \,\sum_{i=1}^d\left( (1+\alpha_i) \lambda_i + \sin \lambda_i \right ).
$$
This is a uniformly elliptic operator, with ellipticity constants $\lambda = \inf \alpha_i$ and $\Lambda = \sup \alpha_i + 1$. We compute
$$
	F^{*}(M) = (1+\alpha_i) \lambda_i,
$$
i.e., modulus a change of coordinates, $F^{*}$ is the laplacian operator. 
\end{example}

%%%%%%%%%%%%%

\section{Some geometric-measure tools} \label{sct tools}

In this section, we gather few tools and elements involved in the proof of Theorem \ref{w2ptheorem}. Throughout, we revisit the by now well-established a priori $W^{2,\delta}$ estimates for solutions to \eqref{eq1}.
% and present a series of building block lemmas of the $W^{2,\delta}$ regularity theory which will play an important role in the proof of Proposition \ref{prop1}. 
Although these are well-known facts, they are presented here for completeness. Most of the proofs are omitted in what follows; the reader is referred to \cite{MR1005611} and \cite[Chapter 7]{ccbook}, where the complete arguments are presented.

\begin{Proposition}[Proposition 7.4, \cite{ccbook}]\label{prop2}Assume that $u\in S(f)$ in $B_1$ and  A\ref{a1} and A\ref{a3} hold. Then, for some universal number $\delta>0$, we have $u\in W^{2,\delta}(B_{1/2})$ and 
$$
	\left\|u\right\|_{W^{2,\delta}(B_{1/2})}\leq C\left(\left\|u\right\|_{L^\infty(B_1)}+\left\|f\right\|_{L^p(B_1)}\right),
$$
where $C>0$ is also a universal constant.
\end{Proposition}

Proposition \ref{prop2} is paramount in the study of a decay rate for the measure of a certain family of subsets of $B_1$. We proceed by presenting a definition.

\begin{Definition}[Chapter 7, pp. 62, \cite{ccbook}]\label{ganda}
Let $\Omega\subset\Rr^d$ be a bounded domain and $u\in\mathcal{C}(\Omega)$. For $M>0$ and $H\subset\Omega$, we define
\begin{align*}
	\underline{G}_{M}(u,H)=\left\{ x_0\in H\,|\,\exists P\,\mbox{concave paraboloid touching $u$ from below at $x_0$}\right\}.
\end{align*}
Analogously, we can define $\overline{G}_{M}$:
\begin{align*}
	\overline{G}_{M}(u,H)=\left\{ x_0\in H\,|\,\exists P\,\mbox{convex paraboloid touching $u$ from above at $x_0$}\right\}.
\end{align*} 
In addition, we have 
$$
	\underline{A}_M(u,H)\doteq H\setminus \underline{G}_M(u,H)$$and $$\overline{A}_M(u,H)\doteq H\setminus \overline{G}_M(u,H).
$$
Finally, 
$$
	G_M(u,H)=\underline{G}_{M}(u,H)\cap\overline{G}_{M}(u,H),
$$
and
$$
	A_M(u,H)=\overline{A}_M(u,H)\cup \underline{A}_M(u,H).
$$
\end{Definition}

To prove that $D^2u$ belongs to  $L^p(B_{1/2})$, we have to study the summability of 
$$
	\sum_{k=1}^\infty M^{pk}\left|A_{M^k}(B_{1/2})\right|.
$$ 
This is due to the following well-known result on measure theory:

\begin{Lemma}\label{lemma1}Let g be a nonnegative and measurable function on $\Omega$ and $\mu_g$ be its distribution function, i.e.,
$$
	\mu_g(t)\,=\,\left|\left\{x\in\Omega\,|\,g(x)\,>\,t\right\}\right|,\;\;\;\;\;\;t>0.
$$ 
Let $\eta>0$ and $M>1$ be constants. Then, for $0<p<\infty$,
$$
	g\in L^p(\Omega)\,\Longleftrightarrow\,\sum_{k=1}^\infty M^{pk}\mu_g\left(\eta M^k\right):=S<\infty
$$
and 
$$
	C^{-1}S\leq\left\|g\right\|_{L^p(\Omega)}^p\leq C\left(\left|\Omega\right|+S\right),
$$
where $C\,>\,0$ is a constant depending only on $\eta$, $M$ and $p$.
\end{Lemma}
\begin{proof}
For the proof, we refer the reader to \cite[Lemma 7.3]{ccbook}.
\end{proof}

\begin{Lemma}\label{lemma2} Assume that $\Omega$ is a bounded domain such that $B_{6\sqrt{d}}\subset\Omega$. Let $u$ be a continuous function in $\Omega$, with $\left\|u\right\|_{L^\infty(\Omega)}\leq 1$, $u\in\overline{S}(f)$ in $B_{6\sqrt{d}}$. Assume further that $\left\|f\right\|_{L^p(B_{6\sqrt{d}})}\leq \delta_0$. Then, $$\left|\underline{G}_M(u,\Omega)\cap Q_1\right|\,\geq\,1-\sigma,$$where $0<\,\sigma\,<1$, $\delta_0>0$ and $M>1$ are universal constants.
\end{Lemma}
%\begin{proof}

The proof of Lemma \ref{lemma2} follows closely the one put forward in \cite[Lemma 7.5]{ccbook} and is omitted here.

\begin{Lemma}\label{lemma76}
Let $\Omega$ be a bounded domain such that $B_{6\sqrt{d}}\subset\Omega$. Assume that $u$ is a continuous functions in $\Omega$ such that $u\in\overline{S}(f)$ in $B_{6\sqrt{d}}$. Assume further that $\left\|f\right\|_{B_{6\sqrt{d}}}\leq\delta_0$ and $\underline{G}_1(u,\Omega)\cap Q_3\neq\emptyset$. Then $$
	\left|\underline{G}_M(u,\Omega)\cap Q_1\right|\geq 1-\sigma,
$$
where $0<\sigma<1$, $\delta_0$ and $M>1$ are universal constants.
\end{Lemma}

As before, we omit the proof of Lemma \ref{lemma76} and refer the reader to \cite[Lemma 7.6]{ccbook}. Next, we recall an elementary result derived from the Calder\'on-Zygmund cube decomposition:

\begin{Lemma}\label{lemma42}
Let $A\subset B\subset Q_1$ be measurable sets and assume that $0<\delta<1$ is such that
\begin{enumerate}
\item $\left|A\right|\,\leq\,\delta;$

\item if $Q$ is a dyadic cube such that $\left|A\cap Q\right|>\delta\left|Q\right|$, then $\tilde{Q}\subset B$, where $\tilde{Q}$ is the predecessor of $Q$.
\end{enumerate}
Then, we have 
$$
	\left|A\right|\,\leq\,\delta\left|B\right|.
$$
\end{Lemma}

\begin{Lemma}\label{lemma77}
Under the hypotheses of Lemma \ref{lemma2}, extend $f$ by zero outside $B_{6\sqrt{d}}$ and let
\[
	A\,\doteq\,\underline{A}_{M^{k+1}}(u,\Omega)\cap Q_1,
\]
\[
	B\,\doteq\, \left(\underline{A}_{M^{k}}(u,\Omega)\cap Q_1\right)\cap\left\lbrace x\in Q_1\,|\,m(f^n)(x)\geq(c_1M^k)^n\right\rbrace.
\]
Then, 
\[
	\left|A\right|\,\leq\,\sigma\left|B\right|,
\]where $0<\sigma<1$, $\delta_0$, $M>1$ and $c_1>0$ are universal constants.
\end{Lemma}

The proof of Lemma \ref{lemma77} relies on standard arguments, together with an auxiliary function $\tilde{u}$ determined by
\begin{equation}\label{utilde}
\tilde{u}(y)\,\doteq\,\frac{2^{2i}}{M^k}u\left(x_0\,+\,\frac{1}{2^i}y\right).
\end{equation} Later in the present paper, we refer to \eqref{utilde}.

In the next lemma, we recall a result on the decay rate of the measure of certain sets.

\begin{Lemma}\label{lemma78}
Under the hypotheses of Lemma \ref{lemma2}, 
\begin{equation}\label{lem78}
	\left|\underline{A}_t(u,\Omega)\cap Q_1\right|\leq c_2t^{-\mu},\,\,\,\,\,\forall\;t\,>\,0,
\end{equation}
where $c_2\,\mu>0$ are universal constants. If, in addition, $u\in S(f)$ in $B_{6\sqrt{d}}$, then
\begin{equation}\label{eq782}
	\left|A_t(u,\Omega)\cap Q_1\right|\leq c_2t^{-\mu},\,\,\,\,\,\forall\;t\,>\,0.
\end{equation}
\end{Lemma}
%\begin{proof}
%Set 
%$$
%	\alpha_k\doteq \left|\underline{A}_{M^k}(u,\Omega)\cap Q_1\right|\;\;\;\;\mbox{and}\;\;\;\;\beta_k\doteq \left|\left\lbrace x\in Q_1\,|\,m(f^d)(x)\geq \left(c_1M^k\right)^d\right\rbrace\right|.
%$$
%Lemma \ref{lemma77} yields $\alpha_{k+1}\leq \sigma(\alpha_k+\beta_k)$, and, therefore,
%$$
%	\alpha_k\leq \sigma^k+\sum_{i=0}^{k-1}\sigma^{k-i}\beta_i.
%$$
%Since 
%$$
%	\left\|f\right\|_{L^1}\leq\delta_0
%$$ 
%and the maximal operator is of weak type $(1,1)$, it follows that 
%$$
%	\beta_k\leq CM^{kd},
%$$ 
%for some constant $C>0$, depending only on $d$ and $\delta_0$. Hence,
%$$
%	\sum_{i=0}^{k-1}\sigma^{k-i}\beta_i\leq C\sum_{i=0}^{k-1}\sigma^{k-i}M^{-dk}\leq Ckm_o^k,
%$$
%where $m_0\doteq\max\left\lbrace\sigma,\,M^{-d}\right\rbrace<1.$By combining these, one obtains 
%$$
%	\alpha_k\leq \sigma^k+Ckm_0^k\leq (1+Ck)m_0^k,
%$$
%which establishes \eqref{lem78}.

%To verify \eqref{eq782} we apply \eqref{lem78} to $-u$; hence, $$\left|\overline{A}_t(u,\Omega)\cap Q_1\right|\leq c_2t^{-\mu},$$which, combined with \eqref{lem78} concludes the proof.
%\end{proof}

We conclude this section by combining the above building-block lemmas, yielding a proof of Proposition \ref{prop2}.

\begin{proof}[Proof of Proposition \ref{prop2}.]
We recur to a standard covering argument and assume that $u\in S(f)$ in $B_{6\sqrt{d}}$, with $\left\|u\right\|_{L^\infty(B_{6\sqrt{d}})}\leq 1$ and $\left\|f\right\|_{L^d(B_{6\sqrt{d}})}\leq \delta_0$. Set $\delta\doteq\mu/2$.
Because of \eqref{eq782}, we have 
$$
	\sum_{k\geq 1}M^{\frac{\mu}{2}k}\left|A_{M^k}(B_{6\sqrt{d}})\cap Q_1\right|\leq C,
$$
for a universal constant $C>0$, where we have used $\Omega = B_{6\sqrt{d}}$. Furthermore, 
$$
	A_{M^k}(u,B_{1/2})\subset A_{M^k}(u,B_{6\sqrt{d}})\cap Q_1;
$$
therefore, 
$$
	\sum_{k\geq 1}M^{\frac{\mu}{2}k}\left|A_{M^k}(B_{1/2})\right|\leq C.
$$

Define $\Theta$ by
\[
\Theta(x):=\inf\left\{M\,|\,x\in G_M(B_{1/2})\right\}.
\] 
Hence, 
\begin{equation}\label{eq79}
	\mu_\Theta(t)\,\leq\,\left|A_t(B_{1/2})\right|.
\end{equation}
The inequality in \eqref{eq79} and the Lemma \ref{lemma1} then lead to 
$$
	\left\|D^2u\right\|_{L^{\frac{\mu}{2}(B_{1/2})}}\leq C\left\|\Theta\right\|_{L^{\frac{\mu}{2}(B_{1/2})}}\leq C.
$$

\end{proof}

\section{A new geometric tangential path}\label{gtm}

Here we find a linking path from the regularity theory of the original operator $F$ to the one of the recession operator $F^{*}$. From the heuristics of the geometric tangential methods, this fact allows us to import the estimates available for $F^{*}$ back to the original operator through this chosen path. As expected, this pass-through mechanism is not without a cost; if the auxiliary problem has a given property, the original one inherits a slightly weaker variation of it. However, as we will see in section \ref{sct proof main thm}, this is just enough to obtain sharp $W^{2,p}$ estimates for \eqref{eq1}.

\begin{Lemma}[Local uniform convergence]\label{luc}
Let $F$ be a uniformly elliptic operator and assume $F^{*}$ exists. Then, for every $\epsilon>0$ there is $\mu_0>0$ such that, for every $\mu<\mu_0$ we have
$$
	\left | \mu F(\mu^{-1}M)\,-\,F^*(M)\right | \,\leq\,\epsilon\left(1+\left\|M\right\|\right),
$$
for every $M\in\mathcal{S}(d)$.  
\end{Lemma}
\begin{proof} The proof of Lemma \ref{luc} is elementary, but we carry it out as a courtesy to the readers. Since all $F_\mu$ are uniformly elliptic operators, by the Arzel\`a-Ascoli Theorem, up to a subsequence, $F_\mu$ converges uniformly in every compact sets of $\mathcal{S}(d)$ to $F^{*}$.  Hence, given $\epsilon>0$ there exists a $\delta>0$ so that 
\[ 
	\left | \mu F(\mu^{-1} M) - F^*(M) \right | \leq \epsilon, 
\]
for all matrices $M$ such that $\|M\|\leq 1$ and all $\mu < \delta$. Now let $M$ be a matrix with $\|M\| > 1$. For any $\mu <  \delta$, we look at 
$$
	\mu_1 = ||M||^{-1} \mu < \mu < \delta
$$ 
and from above,
\[ 
	\left | {\mu_1 F\left( \mu_1^{-1} \frac M {||M||} \right) - F^*\left( \frac M {||M||} \right)} \right | \leq \epsilon. 
\]
Since $\mu_1^{-1} \frac M {||M||} = \mu^{-1} M$, and using the fact that $F^*$ is homogeneous of degree one, we conclude the proof of the lemma.
\end{proof}

Next Lemma is instrumental in ensuring that solutions of \eqref{eq1} can be approximated by $F^*$-harmonic functions, in certain compact functional spaces.

%\begin{Definition}[$\mathcal{C}^{1,1}$-interior estimates]
%We say that $F_\mu(D^2w)=0$ has $\mathcal{C}^{1,1}$-interior estimates in $\Omega$, with constant $c$, if for any $w_0\in\mathcal{C}(\partial %\Omega)$ there exists $w\in\mathcal{C}^2(\Omega)\cap\mathcal{C}(\overline{\Omega})$, a solutioin of
%\begin{equation*}
%\begin{cases}
%F(D^2w)=0\;\;\;& x\in \Omega\\
%w(x)=w_0(x)\;\;\;& x\in\partial \Omega,
%\end{cases}
%\end{equation*}such that $w\in\mathcal{C}^{1,1}\left(\overline{\Omega_0}\right)$ with %$$\left\|w\right\|_{\mathcal{C}^{1,1}(\overline{\Omega_0})}\,\leq\,c\left\|w_0\right\|_{L^\infty(\partial \Omega)},$$for every %$\Omega_0\subset\subset\Omega$.
%\end{Definition}We observe that $F^*(D^2u)=f(x)$ has $C^{1,1}$ estimates, for $F^*$ is concave. In the sequel, we put forward a $W^{2,p}$ theory for solutions of \eqref{eq1} that explores precisely this fact. That is, 

\begin{Lemma}[Approximation Lemma]\label{approx}
Let $u\in\mathcal{C}(B_1)$ be a viscosity solution of 
\begin{equation}\label{eqap}
	F_\mu(D^2u)\,=\,f(x)\;\;\;\;\mbox{in}\;\;B_1
\end{equation}
with $\left\|u\right\|_{L^\infty(B_1)}\leq 1$, and assume that A\ref{a1}-\ref{a3} are satisfied. Given $\delta>0$, there exists $\epsilon=\epsilon(\delta,\,d, F)$ such that, if 
$$
	\left\|f\right\|_{L^p(B_1)}\,\leq\, \epsilon\;\;\;\;\;\;\mbox{and}\;\;\;\;\;\;\mu\,<\,\epsilon,
$$
there exists $h\in \mathcal{C}^{1,1}(B_{3/4})$, solution to
\begin{equation}\label{eqaph}
	F^*(D^2h)\,=\,0\;\;\;\;\mbox{in}\;\;B_{3/4},
\end{equation}
and 
$$
	\left\|u-h\right\|_{L^\infty(B_{1/2})}\leq \delta.
$$
\end{Lemma}

\begin{proof}
We prove the lemma by the way of contradiction. Suppose its statement is false. Then, there exists $\delta_0>0$ and a sequence of functions $(u_j)_{j\in\mathbb{N}}$ and $(f_j)_{j\in\mathbb{N}}$ satisfying 
\[
	F_{\mu_j}(D^2u_j)\,=\,f_j(x)\;\;\;\;\mbox{in}\;\;B_1,
\]
where $\mu_j,\,\left\|f_j\right\|_{L^p(B_1)}=\mbox{o}(1)$ and such that 
\begin{equation}\label{enemy}
	\left\|u_j\,-\,h\right\|_{L^\infty(B_{1/2})}\,>\,\delta_0,
\end{equation}
for every $h$ solution of \eqref{eqaph}. 

On the other hand, because of standard results in elliptic regularity theory, we have that $(u_j)_{j\in\mathbb{N}}\in\mathcal{C}^{0,\alpha}(B_1)$. It means that, through a subsequence if necessary,
$$
	u_j\to u_\infty \quad \text{in the } C_\text{loc} ^{0,\alpha}(B_1) \text{ topology}.
$$
 Furthermore, $F_{\mu_j}$ converges uniformly in compact sets of $\mathcal{S}(d)$ to $F^*$ and $f_j\to 0$. By classical  stability results in the theory of viscosity solutions,
$$
	F^*(D^2u_\infty)\,=\,0\;\;\;\mbox{in}\;B_{3/4}.
$$
Setting $h\equiv u_\infty$ in \eqref{enemy} one is led to a contradiction for $j$ sufficiently large, and the proof is concluded.
\end{proof}

\begin{Remark}
In the proof of Lemma \ref{approx} we have used a compactness argument based on the fact that $(u_j)_{j\in\mathbb{N}}\subset\mathcal{C}^{0,\alpha}(B_1)$ for some $\alpha\in(0,1)$. Though this inclusion suffices for the purposes of the lemma, one can, in fact, establish better regularity results; see for example \cite{SilvTeix}.
\end{Remark}

\begin{Remark}\label{remarkclass}
We notice that 
\[
	u\,-\,h\,\in\,S\left(\frac{\lambda}{n},\Lambda,f(x)-F(D^2h)\right);
\]
it follows from the fact that $h\in\mathcal{C}^{1,1}$; see \cite[Proposition 2.13]{ccbook}.
\end{Remark}

\section{Proof of Theorem \ref{w2ptheorem}} \label{sct proof main thm}

In the sequel, we pursue a finer geometric-measure analysis of the sets $G_M$. This builds upon the approximation lemma to prove Proposition \ref{prop1}, yielding $W^{2,p}$-regularity for solutions of \eqref{eq1}. Apart from the geometric tangential insight, the analysis carried in this section follows closely the one developed by Caffarelli in \cite{MR1005611}, see also \cite[Chapter 7]{ccbook}.

\begin{Lemma}\label{lemma710}
Let $u\in\mathcal{C}(B_1)$ be a viscosity solution of $F(D^2u)=f(x)$ in $B_{3/4}$ such that $\left\|u\right\|_{L^\infty(B_{3/4})}\leq 1$ and $-|x|^2\leq u(x)\leq |x|^2$ in $B_1\setminus B_{3/4}$. Assume further that the assumptions of Lemma \ref{approx} hold true. Then, 
$$
	\left|G_M(u,B_1)\cap Q_1\right|\geq 1-\rho,
$$
where $M>0$ depends only on $d$ and the choice of $\epsilon$ in Lemma \ref{approx} will be determined by $\rho$.
\end{Lemma}
\begin{proof}
Consider the function $h$, from Lemma \ref{approx}, restricted to ${B_{1/2}}$. From standard results in elliptic regularity theory we have $h\in\mathcal{C}^2(\overline{B}_{1/2})$ and $\left\|h\right\|_{\mathcal{C}^{1,1}(\overline{B}_{1,2})}\leq c(d)c_e$. Extend $h$ outside $\overline{B}_{1/2}$ continuously, such that $h=u$ in $B_1\setminus B_{3/4}$ and $\left\|u-h\right\|_{L^\infty(B_1)}=\left\|u-h\right\|_{L^\infty(B_{3/4})}$. Recall that $\left\|u\right\|_{L^\infty(B_{3/4})}=\left\|h\right\|_{L^\infty(B_{3/4})}$.

It follows that $\left\|u-h\right\|_{L^\infty(B_{1/2})}\leq 2$; therefore, $$-2-|x|^2\leq h(x)\leq |x|^2+2$$in $B_1\setminus B_{1/2}$. Hence, there exists $N>0$ so that
\begin{equation}\label{724}
	Q_1\subset G_N(h,B_1).
\end{equation}
Set 
$$
	\omega\doteq\rho_0(u-h),
$$
where $\rho_0>0$ is a constant to be determined later, depending only on the hypotheses of Lemma \ref{approx}. Hence, $\omega$ satisfies the assumptions of Lemma \ref{lemma2}. By applying Lemma \ref{lemma78} one obtains 
$$
	\left|A_t(\omega,B_1)\cap Q_1\right|\leq c_2t^{-\mu},
$$ 
for every $t>0$. Therefore, 
$$
	\left|A_s(u-h,B_1)\cap Q_1\right|\leq C\rho^\mu c_2s^{-\mu},
$$
for some constant $C>0$. It follows that 
$$
	\left|G_N(u-h,B_1)\cap Q_1\right|\geq 1-C_1\rho^\mu\geq 1-\rho_0,
$$
by choosing appropriate constants in Lemma \ref{approx}. Finally, because of \eqref{724}, we get 
$$
	\left|G_{2N}(u,B_1)\cap Q_1\right|\geq 1-\rho_0.
$$
\end{proof}

\begin{Lemma}\label{lemma711}
Let $u\in\mathcal{C}(\Omega)$ be a viscosity solution of $F(D^2u)=f(x)$ in $B_{8\sqrt{d}}$. Assume further that the assumptions of Lemma \ref{approx} hold true. Then, 
$$
	G_1(u,\Omega)\cap Q_3\neq\emptyset \Longrightarrow \left|G_M(u,\Omega)\cap Q_1\right|\geq 1-\rho,
$$
where $M$ and $\rho$ are as in Lemma \ref{lemma710}.
\end{Lemma}
\begin{proof}
Take $x_1\in G_1(u,\Omega)\cap Q_3$; then, there exists an affine function $L_1$ so that 
$$
	-\frac{\left|x-x_1\right|^2}{2}\,\leq\,u(x)-L_1(x)\,\leq\,\frac{\left|x-x_1\right|^2}{2}
$$
in $\Omega$. Define 
$$
	v(x)\,\doteq\,\frac{u(x)-L_1(x)}{c(d)}
$$
and take $c(d)$ large enough so that $\left\|v\right\|_{L^\infty(8\sqrt{d})}\leq 1$ and 
$$
	-|x|^2\leq v(x)\leq |x|^2
$$
in $\Omega\setminus B_{6\sqrt{d}}$. By noticing that $v$ solves 
$$
	\frac{1}{c(d)}F(c(d)D^2v)=\frac{f(x)}{c(d)}
$$ 
in $B_{8\sqrt{d}}$ and recurring to Lemma \ref{lemma710}, one obtains $\left|G_M(v,\Omega)\cap Q_1\right|\geq 1-\rho$. Therefore, $\left|G_{c(d)M}(u,\Omega)\cap Q_1\right|\geq 1-\rho$.
\end{proof}

\begin{Lemma}\label{lemma712}
Let $0<\epsilon_0<1$ and $u$ be a viscosity solution of $F_\mu(D^2u)=f(x)$ in $B_{8\sqrt{d}}$. Assume that 
$$
	\left\|u\right\|_{L^\infty(B_{8\sqrt{d})}}\leq \epsilon\;\;\;\;\;\;\mbox{and}\;\;\;\;\;\;\left\|f\right\|_{L^d(B_{8\sqrt{d})}}\leq \epsilon,
$$
where $\mu,\,\epsilon<\epsilon_0$. Extend $f$ by zero outside $B_{8\sqrt{d}}$ and set 
$$
	A\doteq A_{M^{k+1}}(u,B_{8\sqrt{d}})\cap Q_1,
$$
$$
	B\doteq (A_{M^k}(u,B_{8\sqrt{d}})\cap Q_1)\cup\left\lbrace x\in Q_1\,|\,m(f^d)(x)\geq (c_3M^k)^d \right\rbrace.
$$
Then, 
$$
	|A|\leq \epsilon_0|B|,
$$
where $M>1$ depends only on $d$ and $c_e$ and $c_3,\,\epsilon>0$ depend only on $d$, $c_e$, $\epsilon_0$, $\lambda$ and $\Lambda$.
\end{Lemma}
\begin{proof}
The proof is based on the Lemma \ref{lemma42}. Notice that $|u|\leq 1\leq |x|^2$ in $B_{8\sqrt{d}}\setminus B_{6\sqrt{d}}$. Hence, by setting $\Omega\equiv B_{8\sqrt{d}}$ in Lemma \ref{lemma710}, one obtains 
$$
	\left|G_{M^{k+1}(u,B_{8\sqrt{d}})\cap Q_1}\right|\,\geq\,\left|G_M(u,B_{8\sqrt{d}})\right|\,\geq\,1-\rho.
$$ 
To conclude the proof, we have to check that if $Q$ is a dyadic cube of $Q_1$ such that 
\begin{equation}\label{726}
	\left|A_{M^{k+1}}(u,B_{8\sqrt{d}})\cap Q\right|\,\geq\,\left|A\cap Q\right|\,\geq\,\rho\left|Q\right|,
\end{equation}
then $\widetilde{Q}\subset B$. Suppose not, i.e., there exists $x_1\in\widetilde{Q}$ so that
\begin{equation}\label{727}
	x_1\,\in\,\widetilde{Q}\cap G_{M^k}(u,B_{8\sqrt{d}}),
\end{equation} 
and
\begin{equation}\label{728}
	m\left(f^d\right)(x_1)\,\leq\,\left(c_3M^k\right)^d.
\end{equation}
Before we proceed, we notice that if $Q$ is a dyadic cube of $Q_1$, we have $Q=Q_{1/2^i}(x_0)$ for some $i\geq 0$ and $x_0\in Q_1$. Define $\tilde{u}$ by \eqref{utilde} and let $\widetilde{\Omega}$ 
%as in Lemma \ref{lemma77}, 
be the image of $\Omega$ by 
\begin{equation}\label{trans}
	x\,=\,x_0\,+\,\frac{1}{2^i}y,
\end{equation}
where we take $\Omega$ to be $B_{8\sqrt{d}}$. It remains to check that $\tilde{u}$ satisfies the hypotheses of Lemma \ref{lemma711}.
Because $B_{8\sqrt{d}/2^i}(x_0)\subset B_{8\sqrt{d}}$, one has $B_{8\sqrt{d}}\subset \widetilde{\Omega}$. Hence, $\tilde{u}$ is a viscosity solution of 
$$
	G(D^2\tilde u)\,=\,\tilde{f},$$where $$G(D^2w)=\frac{1}{M^k}F\left(M^kD^2w\right),
$$
and 
$$
	\tilde{f}(y)=\frac{1}{M^k}f\left(x_0+\frac{1}{2^i}y\right).
$$

Notice that $\left|x_1-x_0\right|_\infty\leq 3/2^{i+1}$ yields $B_{8\sqrt{d}/2^i}(x_0)\subset Q_{19\sqrt{d}/2^i}(x_1)$. Hence, \eqref{728} implies
\[
	\left\|\tilde{f}\right\|^d_{L^d(B_{8\sqrt{d}})}\leq \frac{2^{2i}}{M^{kd}}\int_{Q_{19\sqrt{d}/2^i}(x_1)}\left|f(x)\right|^ddx\leq C_dc_3^d;
\]
by taking $c_3$ small enough, it follows that 
\[
	\left\|\tilde{f}\right\|_{L^d(B_{8\sqrt{d}})}\leq \rho.
\]
Finally, because of \eqref{727}, we also have $G_1(\tilde{u},\widetilde{\Omega})\cap Q_3\neq\emptyset$. Therefore,
\[
	\left|G_M(\tilde{u},\widetilde{\Omega})\cap Q_1\right|\,\geq\, (1-\rho)|Q_1|,
\]
that is,
\[
	\left|G_{M^{k+1}}(u,B_{8\sqrt{d}})\cap Q\right|\,\geq\,(1-\rho)|Q|,
\]
which leads to a contradiction in face of \eqref{726}, finishing the proof.
\end{proof}

Next we present the proof of Proposition \ref{prop1}.

\begin{proof}[Proof of Proposition \ref{prop2}]
Let $M$ be as in Lemma \ref{lemma712} and take $\rho$ such that $$\rho M^p\,=\,\frac{1}{2}.$$ For $k\geq 0$, set 
\[
	\alpha_k\doteq\left|A_{M^k}(u,B_{8\sqrt{d}})\cap Q_1\right|,\;\;\;\beta_k\doteq\left|\left\lbrace x\in Q_1\,|\,m(f^d)(x)\geq\left(c_3M^k\right)^d\right\rbrace\right|.
\]
Lemma \ref{lemma712} implies that $\alpha_{k+1}\leq \rho(\alpha_k+\beta_k)$, which leads to 
\begin{equation}\label{729}
	\alpha_k\,\leq\,\rho^k\,+\,\sum_{i=0}^{k-1}\rho^{k-i}\beta_i.
\end{equation}
Clearly, $f^d\in L^{p/d}$; hence, $m(f^d)\in L^{p/d}$ and 
$$
\left\|m(f^d)\right\|_{L^{p/d}}\leq C\left\|f\right|^d_{L^p}\leq C.
$$
Therefore, by Lemma \ref{lemma1}, we have
\begin{equation}\label{730}
	\sum_{k\geq 0}M^{pk}\beta_k\,\leq\, C.
\end{equation}

On the other hand, we have 
$$
	\mu_\Theta(t)\,\leq\,\left|A_t(B_{1/2})\right|\,\leq\,\left|A_t(B_{1/2})\cap Q_1\right|.
$$ 
By recurring once again to Lemma \ref{lemma1}, it suffices to verify that
\[
	\sum_{k\geq 1}M^{pk}\alpha_k\,\leq\,C,
\]
which follows from \eqref{729} and \eqref{730}, as follows:
\begin{align*}
	\sum_{k\geq 1}M^{pk}\alpha_k&\leq\sum_{k\geq 1}(\rho M^p)^k\,+\,\sum_{k\geq 1}\sum_{i=0}^{k-1}\rho^{k-i}M^{p(k-i)}M^{pi}\beta_i\\
	&=\sum_{k\geq 1}2^{-k}+\left(\sum_{i\geq 0}M^{pi}\beta_i\right)\left(\sum_{j\geq 1}2^{-j}\right)\leq C.
\end{align*}
\end{proof}
%%%%%%
\section{Estimates for variable coefficient equations} \label{sct var coeff}

In this section, we comment on a generalization of Theorem \ref{w2ptheorem} for fully nonlinear elliptic operators with variable coefficients, $F \colon B_1\times \mathcal{S}(d) \to \mathbb{R}$. Herein we assume assumption A\ref{a1} holds uniformly in $x\in B_1$, that is, for any $x \in B_1$ and any $M\in\mathcal{S}(d)$, there holds
$$
	\lambda\left\|N\right\|\,\leq\,F(x, M+N)\,-\,F(x, M)\,\leq\,\Lambda\left\|N\right\|,
$$
for every $N\,\geq\,0$. We also assume $F(x,0) = 0$. Following Caffarelli \cite{MR1005611}, we measure the oscillation of the coefficients of an elliptic operator $F \colon B_1\times \mathcal{S}(d) \to \mathbb{R}$ around a point $x_0 \in B_1$ by
$$
	\beta_F(x,x_0) := \sup\limits_{M\in \mathcal{S}(d)\setminus \{0\}} \dfrac{\left |F(x, M) - F(x_0, M) \right |}{\|M\|}.
$$	
Now, for any $\mu>0$ and any $(x,M) \in B_1\times \mathcal{S}(d)$, we define 
$F_\mu(x,M) := \mu F(x, \mu^{-1}M)$ and the recession function
$$
	F^{*}(x,M) := \lim\limits_{\mu \to 0} F_\mu (x,M),
$$ 
provided it exists for all pair $(x,M) \in B_1\times \mathcal{S}(d)$. While it is plain to check that
$$
	\beta_F(x,x_0) = \beta_{F_\mu}(x,x_0),
$$
the oscillation of the recession function can indeed be smaller than the original one. This is the case, for instance, when the operator satisfies $F(x,M) = G(M)$, for all $M \in \mathcal{S}(d) \setminus B_R$ (see for instance construction from section \ref{sct density}). Indeed, the recession function $F^{*}(x,M)$ is simply  $G^{*}(M)$; a constant coefficient operator. 

\begin{teo}\label{w2,p for var coeff} Let $F\colon B_1\times \mathcal{S}(d) \to \mathbb{R}$ be uniformly elliptic, $p>d$, and $u$ be a viscosity solution of 
$$
	F(x, D^2u)\,=\,f(x), \quad \text{in } B_1.
$$ 
Assume $F^{*}(x,M)$ is uniformly convex with respect to $M$ and for some constants $\alpha > 0$ and $C>0$, there holds
\begin{equation}\label{hyp calpha coeff}
	\intav{B_r} \beta_{F^{*}}(x, x_0)^d dx \le Cr^{d\alpha},
\end{equation}
for $r\le r_0 \ll 1$ and all $x_0 \in B_1$. Then $u\in W^{2,p}(B_{1/2})$ and there exists $C>0$ so that
\[
	\left\|u\right\|_{W^{2,p}(B_{1/2})}\,\leq\,C\left(\left\|u\right\|_{L^{\infty}(B_{1})}\,+\,\left\|f\right\|_{L^{p}(B_{1})}\right).
\]
\end{teo}
\begin{proof}
	The proof of Theorem \ref{w2,p for var coeff} follows closely the reasoning from section \ref{sct proof main thm}. The key observation is that the version of Lemma \ref{approx} for operators with variable coefficients can be proven by similar analysis. Indeed,  if $u\in\mathcal{C}(B_1)$ is a viscosity solution of $F_\mu(x,D^2u)\,=\,f(x)$ in $B_1$, with $\left\|u\right\|_{L^\infty(B_1)}\leq 1$, then given $\delta>0$, there exists $\epsilon>0$ such that, if 
$$
	\left\|f\right\|_{L^p(B_1)}\,\leq\, \epsilon\;\;\;\;\;\;\mbox{and}\;\;\;\;\;\;\mu\,<\,\epsilon,
$$
we can find a function $h\in \mathcal{C}^{2, \, \beta}(B_{3/4})$, satisfying
\begin{equation}\label{eqaph1}
	F^*(x, D^2h)\,=\,0\;\;\;\;\mbox{in}\;\;B_{3/4},
\end{equation}
and 
$$
	\left\|u-h\right\|_{L^\infty(B_{1/2})}\leq \delta.
$$
The same contradiction argument employed in the proof of Lemma \ref{approx} assures the existence of a solution $h$ that is $\delta$-close to $u$. That $h$ is of class  $\mathcal{C}^{2, \, \beta}(B_{3/4})$ follows from hypothesis \eqref{hyp calpha coeff} combined with Caffarelli's Schauder regularity theory, \cite{MR1005611}. Owing to the existence of $h$, as well as its universal regularity theory beyond $C^{1,1}$, one can continue the analysis of section \ref{sct proof main thm}, with minor modifications. For example, as regards the proofs of Lemmas \ref{lemma710} and \ref{lemma711}, these remain the same. Lemma \ref{lemma712} must accommodate condition \eqref{hyp calpha coeff} as an assumption. The geometric-measure arguments from section \ref{sct tools} conclude the proof.
\end{proof}

As a final remark, condition \eqref{hyp calpha coeff} refers to a sort of H\"older continuity of the coefficients in $L^d$ sense. If necessary, one can relax such a hypothesis to weaker continuity assumption on $\beta_{F^{*}}$, namely Dini continuity suffices. We do not want to enter in this issue here other than mentioning that indeed Theorem  \ref{w2,p for var coeff} yields novel information even for convex equations, as the oscillation of the coefficients of the recession function can be strictly less then the original operator.
%%%%%%
\section{A priori BMO type estimates} \label{sct BMO}

In this section, we discuss the borderline case $p=\infty$. It is well established that boundedness of the source function $f$ does not imply, in general, that the hessian is bounded, even if $F$ is the laplacian operator. That is, sharp $W^{2,p}$-regularity estimates fail in the limit case $p=\infty$. 
Recall a function $g$ is said to belong to the $p$-BMO space if
\begin{equation}
	\sup_{\rho>0}\frac{1}{\rho^d}\int_{B_\rho}\left|g(x)- \langle g \rangle_\rho \right|^pdx\,\leq\,C,
\end{equation}
for a constant $C>0$ independent of $\rho$. Hereafter in the paper
$$
	\langle g \rangle_\rho := \intav{B_\rho}{g(x) dx}.
$$

The ultimate goal of this section is to show that solutions of \eqref{eq1} have hessians in $p-$BMO$(B_{1/2})$, for every $p>d$,  provided $f$ is in $p-$BMO$(B_{1})$ and $F=F^{*}$ outside a large ball $B_K \subset \mathcal{S}(d)$. That is, in this section we work under the extra assumption: 

\begin{Assumption}\label{a4} 
	There exists a constant $L\gg 1$, such that $F = F^{*}$ for all $M \in \mathcal{S}(d),$ with $\|M\| \ge L$. The recession function $F^{*}$ has a priori $C^{2,\alpha}$ interior estimates.
\end{Assumption}

Assumptions A\ref{a4} and A\ref{a2} are typical of some geometric PDEs, where convexity of $c$-level sets are verified for $c\gg 1$. This is also the case of special approximating operators, as the ones we will build up in section \ref{sct density}. The main result of this section is then:

\begin{teo} \label{p-BMO est}
Let $u$ be a viscosity solution of \eqref{eq1}, assume  A\ref{a1}-A\ref{a4} are satisfied, and $f\in p-$BMO$(B_1)$, for some $p>d$. Then, $D^2u\in p-$BMO$(B_{1/2})$ and
\[
	\left\|D^2u\right\|_{p-BMO(B_{1/2})}\leq C\left(\left\|u\right\|_{L^\infty(B_1)}\,+\,\left\|f\right\|_{p-BMO(B_1)}\right).
\]
\end{teo}

It is worth commenting that Theorem \ref{p-BMO est} closely relates to, and generalizes to some extent, developments reported in \cite{MR1978880} (c.f. \cite[Theorem A]{MR1978880}).  The proof of Theorem \ref{p-BMO est} will be divided into two steps. Initially we show the existence of {good} approximating paraboloids, $P_\rho$, and in the sequel we apply Theorem \ref{w2ptheorem} for a normalized sequence of functions related to the problem. The analysis here is similar to the one put forward in \cite{SilvTeix} and so we just comment on the necessary modifications. Hereafter constants that depend only on the ellipticity, dimension and  $C^{2,\alpha}$ interior estimates for $F^{*}$ will be called universal. 

\begin{Proposition} \label{l:compactness} Under the assumptions of Theorem \ref{p-BMO est}, there exist two positive universal constants, $\mu_0$ and $r$, such that if $u$ is a solution of $F_\mu(D^2u) = f$, with $\mu + \|f\|_{L^p} \leq \mu_0$ and $||u||_{L^\infty}\leq 1$, then, there exists a paraboloid $P$, with universal controlled norm $\|P\| \leq C$ satisfying  
$$
	\sup\limits_{B_r} |u-P| \leq r^2.
$$
\end{Proposition}

\begin{proof} 
The proof is based on compactness arguments, similar to the one carried out in \cite{MR3158810} and in \cite{SilvTeix}. We omit the details.
\end{proof}

\noindent {\it Proof of Theorem \ref{p-BMO est}} \\
\noindent {\bf Step 1}\, (Existence of approximating quadratic polynomials). Let $u$ be a solution of $F(D^2u) = f$. An appropriate dilatation of $u$, $v(x) := \delta_1 u(\delta_2 x)$ verifies $\|v\|_\infty \le 1$ and solves
$$
	F_\mu (D^2v) = \tilde{f},
$$
with $\|\tilde{f}\|_{p-BMO(B_1)} + \mu \le \mu_0$. The choices for $\delta_1$ and $\delta_2$ depend only on $\|u\|_\infty$, $\|f\|_{p-BMO(B_1)}$ and universal data. We will prove the $p$-BMO estimate for $v$, which clearly gives the corresponding one for $u$. The first step is to show, by finite induction process, the existence of quadratic polynomials
$$
    {P}_k(x) := a_k + \mathbf{b}_k \cdot X + \dfrac{1}{2}x^t M_k x,
$$
verifying
\begin{equation}
    \label{P1} F^*(M_k) = \langle \tilde f \rangle_1, 
\end{equation}
\begin{equation}
    \label{P2} \sup\limits_{B_{r^k}} |v -  {P}_k| \le  r^{2k},  \\
\end{equation}
\begin{equation}
    \label{P3} |a_k - a_{k-1}| + r^{k-1} |\mathbf{b}_k - \mathbf{b}_{k-1}| +  r^{2(k-1)}  |M_k - M_{k-1}|
    \le C r^{2(k-1)}. 
\end{equation}
The constant $r$ appearing in \eqref{P2} and \eqref{P3} is the one from Proposition \ref{l:compactness}. We set  ${P}_0 = {P}_{-1}  = 0$, and the first step $k=0$ is
immediately satisfied. Suppose $k=0,1, \cdots, n$ have been checked, define the new function $v_m
\colon B_1 \to \mathbb{R}$ by
$$
    v_m(X) := \dfrac{(v -  {P}_m)(r^m X)}{ r^{2m}}.
$$
From induction hypothesis, $\|v_m\|_\infty \le 1$ and
$$
   \mu F \left (\mu^{-1} \left( D^2v_m + {M_i}  \right)\right  ) =   \tilde{f}(r^ix).
$$
From uniform convergence, the operators 
$$
	F_m(M) := F(M+M_m) \quad \text{and} \quad F^*_m(M) := F^*(M+M_m),
$$
are uniformly close. Also, since $F^*(M_m) = \langle f \rangle_1$,  from the smallness condition on $\|\tilde{f}\|_{p-BMO}$, the equation 
$$
	F_m^*(D^2 \xi)=  \langle f \rangle_1
$$
satisfies the same a priori $C^{2,\alpha}$-regularity estimate as the original $F^*$ and it is under the assumption of Proposition \ref{l:compactness}. Hence,  we can find a quadratic polynomial $\tilde{P}$ such that
\begin{equation}\label{proof LL eq i}
	\|v-\tilde{P}\|_{L^\infty(B_r)} \leq r^2.
\end{equation}
\smallskip

\noindent {\bf Step 2}\, (BMO estimate). Define $P_{i+1}(X) := P_i(X) + r^{2i}\tilde{P}(r^{-i} X)$ and rescale \eqref{proof LL eq i} back to obtain the $m+1$ step of induction. To conclude, for $\rho>0$, choose $m$ such that $0<r^{m+1} < \rho \le r^m$  and we apply Theorem \ref{w2ptheorem} to $v_m$ as to obtain
 \begin{align*}
	\frac{1}{\rho^d}\int_{B_\rho}\left|D^2v(y)-M_m\right|^pdy\,&\le \dfrac{1}{r} \cdot \dfrac{1}{r^m} \int_{B_{r^m}}\left|D^2v(y)-M_m\right|^p {dy} \\
	&=\,\int_{B_{r^m}}\left|D^2v_m(x)\right|^pdx\\
	&\,\leq C,
\end{align*}
and the proof of Theorem \ref{p-BMO est} is complete. \hfill $\square$

We conclude this section with a remark about the relation of inclusion involving the spaces $p$-BMO, for different values of $1<p< \infty$. Inded, for $1\,<\,p,\,q\,<\,\infty$, one always has
\begin{equation}\label{eqn pqbmo}
	p-BMO(B_1)\,=\,q-BMO(B_1)\;\;\;\;\mbox{and}\;\;\;\;\left\|u\right\|_{p-BMO(B_{1})}\,\propto\,\left\|u\right\|_{q-BMO(B_{1})}.
\end{equation}
This fact is a consequence of the John-Nirenberg inequality: assume $u\in BMO(B_1)$ and let $\rho>0$ be so that $Q_\rho\subset B_1$; then, there exist constants $C>0$ and $\alpha>0$, depending only on the dimension, for which
$$
	\left|\left\lbrace x\in Q_\rho\,:\,\left|u(x)-\left\langle u\right\rangle_\rho\right|>\lambda\right\rbrace\right|\,\leq\,Ce^{-\alpha\frac{\lambda}{\left\|u\right\|_{BMO(B_1)}}}\left|Q_\rho\right|.
$$
See \cite{JNBMO}. This inequality yields 
$$
	\int_{Q_\rho}\left|u(x)-\left\langle u\right\rangle_\rho\right|^pdx\leq p\int_0^\infty\lambda^{p-1}Ce^{-\alpha\frac{\lambda}{\left\|u\right\|		_{BMO(B_1)}}}d\lambda\left|Q_\rho\right|,
$$
which in turn, combined with 
$$
	p\int_0^\infty\lambda^{p-1}Ce^{-\alpha\frac{\lambda}{\left\|u\right\|_{BMO(B_1)}}}d\lambda\leq C\left\|u\right\|_{BMO(B_1)}^p,
$$
establishes \eqref{eqn pqbmo}. Hence the condition $p>d$ in Theorem \ref{p-BMO est} can be removed and the following a priori estimate: 
\[
	\left\|D^2u\right\|_{q-BMO(B_{1/2})}\,\leq\, C\left(\left\|u\right\|_{L^\infty(B_1)}\,+\,\left\|f\right\|_{p-BMO(B_1)}\right),
\]
holds for any $1<p, \, q<+\infty$, c.f. Theorem \ref{p-BMO est} and \cite[Lemma 3]{JNBMO}.
%%%%%%%%%%%%%%%%%%%%%%%%%%%%%%%%%%%%%%%%%
\section{$W^{2,p}$-density in the class of viscosity solutions} \label{sct density}  

In this section, we show that $W^{2,p}$ solutions are dense in the set of  $C^{0}$-viscosity solutions. We start by recalling the formalism of the class of {\it all solutions of all \linebreak % 
$(\lambda, \Lambda)-$uniform elliptic equations.} 

\begin{Definition}\label{S class}
Let $f$ be a function in $B_1$ and $0 < \lambda \le \Lambda $ be positive numbers. We denote by $\underbar{{S}}(\lambda, \Lambda, f)$ the set of all continuous functions $u$ in $B_1$ such that $\mathscr{P}^{+}_{\lambda, \Lambda}(D^2u) \ge f(x)$ in the viscosity sense. Analogously, $\bar{{S}}(\lambda, \Lambda, f)$ denotes the set of all continuous functions $u$ in $B_1$ such that $\mathscr{P}^{-}_{\lambda, \Lambda}(D^2u) \le f(x)$ in the viscosity sense. The class of all continuous viscosity solutions among all $(\lambda, \Lambda)$-elliptic equations is defined as
$$
	 {{S}}(\lambda, \Lambda, f) := \underbar{{S}}(\lambda, \Lambda, f) \cap \bar{{S}}(\lambda, \Lambda, f).
$$
\end{Definition}

Heuristically, a function $u$ belongs to $ {{S}}(\lambda, \Lambda, f)$ if it is a continuous viscosity solution of a variable coefficient equation
$$
	F(x, D^2u) = f(x),
$$
where $F$ is $(\lambda, \Lambda)-$uniform elliptic. A cornerstone result in the theory of fully nonlinear elliptic equations is that functions in  ${{S}}(\lambda, \Lambda, f)$ are {\it universally} H\"older continuous, \cite{ccbook}. It is known that H\"older continuous estimates are the best available for solutions to {\it measurable} coefficient equations. 

Our next result shows that any continuous viscosity solution can be approximated by a $W^{2,p}$-viscosity solution.

\begin{teo}\label{thm density} Let $f\in L^p(B_1)$, $F\colon B_1 \times \mathcal{S}(d) \to \mathbb{R}$ a $(\lambda, \Lambda)$-elliptic operator and $u$ continuous viscosity solution of $F(x,D^2u) = f(x)$ in $B_1$. Given $\delta>0$, there exists a sequence of functions $\{u_j\}_{j\ge 1} \subset  W^{2,p}_\text{loc}(B_1) \cap {{S}}(\lambda - \delta, \Lambda +\delta, f)$ that converges locally uniformly to $u$.
\end{teo} 
\begin{proof}
We construct a sequence of operators $F_j \colon B_1 \times \mathcal{S}(d) \to \mathbb{R}$ as follows: given $\delta>0$, consider the convex (extremal) operator 
\[
	L_\delta(M)\,:=\,\left(\Lambda+\delta \right)\sum_{e_i>0}e_i\,+\, \left ( \lambda - \delta \right ) \sum_{e_i<0}e_i,
\]
where $e_i$ are the eigenvalues of the matrix $M\in\mathcal{S}(d)$. In the sequel we set
$$
	F^j(x,M) := \max \{ F(x,M),  \, L_\delta(M) - C_j\},
$$
where $C_j$ is a divergent sequence of positive numbers to be determined a posteriori. From $(\lambda, \Lambda)$-ellipticity of $F$, we verify that 
$$
	\begin{array}{lll}
		F(x, M) &\ge& \displaystyle \lambda  \sum_{e_i>0}e_i\,+\,\Lambda\sum_{e_i<0}e_i \\
		&\ge &  \displaystyle \lambda  \sum_{e_i>0}e_i\, - \,\Lambda  \|M\| \\ 
		&= & \displaystyle L_\delta(M) - (\Lambda+\delta -\lambda) \sum_{e_i>0}e_i\,-  \, (\lambda - \delta)\sum_{e_i<0}e_i -  \Lambda  \|M\| \\
		&\ge&  \displaystyle L_\delta(M) -  (2\Lambda - \lambda + \delta) \|M\| \\
		&\ge & \displaystyle L_\delta(M) - C_j,
	\end{array} 
$$
provided $\|M\| \le j$ and $C_j := j (2\Lambda - \lambda + \delta)$. This shows that
$$
	F^j = F \text{ in } B_j \subset \mathcal{S}(d).
$$
We now compute the recession function of $F^j$. For that, we look at the tangential path
$$
	 F^j_\mu (x, M) := \mu F^j (x, \mu^{-1}M) = \max \{ F_\mu(x, M), L_\delta(M) - \mu C_j\}.
$$
Since $F_\mu$ is $(\lambda, \Lambda)$-elliptic for all $\mu$, we can estimate
$$
	\begin{array}{lll}
		F_\mu(x, M) &\le& \displaystyle \Lambda  \sum_{e_i>0}e_i\,+\,\lambda\sum_{e_i<0}e_i \\
		& = &  \displaystyle L_\delta(M) - \delta  \sum_{e_i>0}e_i\, + \,\delta\sum_{e_i<0}e_i  \\ 
		& \le & \displaystyle L_\delta(M) - \delta \|M\| \\
		&\le&  \displaystyle L_\delta(M) - \mu C_j,
	\end{array} 
$$
provided $\|M\| \ge \frac{\mu C_j}{\delta}$. In particular, we conclude that $F^j(x,M) = L_\delta(M) - C_j$ outside the ball of radius $\sim C_j$ and that $\left (F^j\right )^{*} = L_\delta$ -- a convex operator, with a priori $C^{2,\alpha}$ estimates due to Evans and Krylov Theorem. 
\begin{figure}[h!]\label{fig density}
{
\begin{pspicture}(0,-4.3091993)(12.99789,4.339199)
\definecolor{color192b}{rgb}{0.8,0.8,0.8}
\definecolor{color985}{rgb}{0.6,0.6,0.6}
\psline[linewidth=0.02cm,arrowsize=0.05291667cm 2.0,arrowlength=1.4,arrowinset=0.4]{<-}(5.303632,3.5808008)(5.303632,-4.299199)
\psline[linewidth=0.02cm,arrowsize=0.05291667cm 2.0,arrowlength=1.4,arrowinset=0.4]{->}(0.0,-0.33919922)(10.700074,-0.33919922)
\psdots[dotsize=0.12](5.3,-1.7791992)
\psline[linewidth=0.04cm](5.32,-1.7991992)(8.44,3.8408008)
\psline[linewidth=0.04cm](5.28,-1.7991992)(0.56,-2.5791993)
\psline[linewidth=0.02cm,linestyle=dotted,dotsep=0.16cm](6.54,0.32080078)(6.54,-0.35919923)
\usefont{T1}{ptm}{m}{n}
\rput(5.661455,-1.9941993){$-C_j$}
\usefont{T1}{ptm}{m}{n}
\rput(3.8714552,-0.11419922){$-j$}
\usefont{T1}{ptm}{m}{n}
\rput(6.491455,-0.5741992){$j$}
\usefont{T1}{ptm}{m}{n}
\rput(8.761456,4.1458006){$L_\delta - C_j$}
\psline[linewidth=0.02cm,linestyle=dotted,dotsep=0.16cm](1.5,-2.4191992)(1.52,-0.33919922)
\usefont{T1}{ptm}{m}{n}
\rput(1.421455,-0.13419922){$\sim -C_j$}
\usefont{T1}{ptm}{m}{n}
\rput(7.941455,-0.5541992){$\sim C_j$}
\rput{-230.0}(12.6404,-3.519572){\pstriangle[linewidth=0.002,linecolor=color192b,dimen=outer,fillstyle=solid,fillcolor=color192b](7.1408014,-1.2272713)(1.9485885,4.829286)}
\rput{-410.0}(2.6867437,1.9324062){\pstriangle[linewidth=0.002,linecolor=color985,dimen=outer,fillstyle=solid,fillcolor=color192b](3.415401,-4.3631196)(1.9693696,4.896905)}
\psline[linewidth=0.02cm](5.32,-1.7991992)(0.6,-2.5791993)
\psline[linewidth=0.02cm,linestyle=dotted,dotsep=0.16cm](3.98,-0.35919923)(3.94,-2.039199)
\psline[linewidth=0.02cm](5.32,-1.7991992)(8.4,3.7608008)
\psline[linewidth=0.02cm,linestyle=dotted,dotsep=0.16cm](8.04,3.0208008)(7.98,-0.2791992)
\usefont{T1}{ptm}{m}{n}
\rput(9.989453,0.96580076){$(\lambda, \Lambda)$-cone of ellipticity}
\psdots[dotsize=0.16](5.3,-0.33919922)
\usefont{T1}{ptm}{m}{n}
\rput(5.3114552,3.7458007){$\mathbb{R}$}
\usefont{T1}{ptm}{m}{n}
\rput(10.751455,-0.5541992){$\mathcal{S}(d)$}
\psbezier[linewidth=0.02](5.32,-0.33919922)(6.011078,0.13768226)(5.959204,0.6094189)(6.4180017,0.7138858)(6.876799,0.8183527)(7.1050534,0.5948251)(7.313234,0.7933435)(7.5214143,0.9918619)(7.0734425,1.939706)(7.4452457,2.2795284)(7.817049,2.619351)(9.02,0.06080078)(8.8975725,2.7408009)
\psbezier[linewidth=0.02](5.26,-0.3791992)(4.6236787,-0.80633855)(4.655462,-1.2980752)(4.1974635,-1.4025421)(3.7394657,-1.507009)(3.511609,-1.2834814)(3.3037918,-1.4819998)(3.0959742,-1.6805182)(3.543165,-2.6283622)(3.1720097,-2.9681847)(2.8008544,-3.3080072)(1.6,-0.7494571)(1.7222143,-3.4294572)
\usefont{T1}{ptm}{m}{n}
\rput(8.961455,2.9458008){$F$}
\psline[linewidth=0.01cm,linestyle=dashed,dash=0.16cm 0.16cm](5.34,-1.7791992)(9.78,-1.0791992)
\psline[linewidth=0.01cm,linestyle=dashed,dash=0.16cm 0.16cm](5.34,-1.7591993)(3.92,-4.219199)
\end{pspicture} 
}
%\vspace{-2cm}
\caption{%
This figure illustrates the construction of the operator $F^j$: the graph of an arbitrary $(\lambda,\Lambda)$-elliptic operator $F$ lies inside the $(\lambda,\Lambda)$-cone of ellipticity. By tilting a little bit the opening of the cone and placing its vertex at a very negative axis point, the boundary of the new cone, which contains the graph of $L_\delta - C_j$, stays below the original $(\lambda,\Lambda)$-cone within $B_j$, and above  in the complement of the ball of radius $\sim C_j$. 
}%
\end{figure}
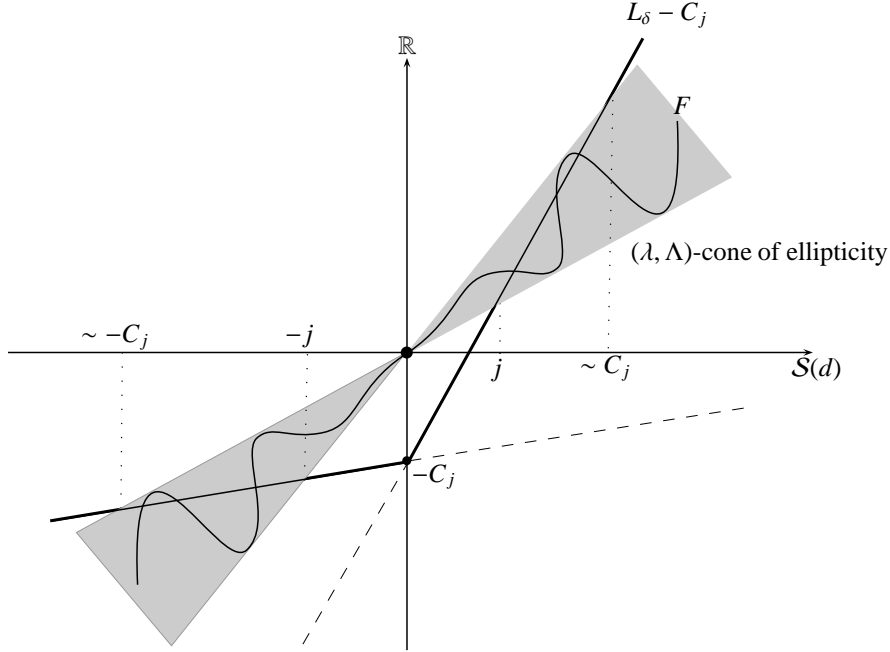

Thus, from Theorem \ref{w2,p for var coeff}, for each $j \ge 1$ fixed, the constructed operator $F^j$ has a priori $W^{2,p}$ interior estimates. That is, there exist constants $\kappa_j \ge 1$ such that, for any viscosity solution $v$ of 
$$
	F^j(x, D^2v) = g(x),
$$
where holds
$$
	\|v\|_{W^{2,p}(B_{1/2})} \le \kappa_j \left ( \|v\|_{L^\infty(B_1)} + \|g\|_{L^p(B_1)} \right ).
$$
Finally, we construct $u_j$ to be the viscosity solution of the Dirichlet problem
$$
	\left \{
		\begin{array}{rlll}
			F^j(x,D^2u_j) &=& f(x) &\text{ in } B_1 \\
			u_j &=& u &\text{ on } \partial B_1.
		\end{array}
	\right.
$$
From the discussion above, each $u_j$ is locally in $W^{2,p}$. Also, since $F^j = F$ within $B_j$, it follows by stability of viscosity solutions and uniqueness of the Dirichlet problem that, up to a subsequence, $u_j \to u$ locally in the $C^{0, \, \alpha}$-topology. The proof of the Theorem is complete.
\end{proof}

We notice that for constant coefficient equations, $F(D^2u) = f(x)$, Theorem \ref{thm density} yields a sequence of $W^{2,p}$ approximating functions that converge in the $C^{1,\alpha}_\text{loc}$-topology. One should compare Theorem \ref{thm density} with the result from \cite{JKS} and also from \cite{Krylov2012, Krylov2013}. 

This corpus of results has an important consequence for the theory of fully nonlinear elliptic PDEs. In fact, an effective tool in the study of viscosity solutions of homogeneous fully nonlinear equations is the mechanism devised by R. Jensen \cite{jensen}, known as inf-sup convolutions, or upper/lower $\epsilon$-envelope. The regularizing effects of this procedure gives semi-concave sub-solutions or semi-convex super-solutions. It now follows from Theorem \ref{thm density} that, when aiming to establish a property closed under uniform limits, one can assume, with no loss, that solutions of non-homogeneous equations, $F(x,D^2u) = f(x)$ are locally of class $W^{2,p}$; setting up a more comfortable starting point for further developments of the theory.

%\bibliography{bib_june_2015}

\begin{thebibliography}{10}

\bibitem{MR3067831}
D.~Ara{\'u}jo and E.~Teixeira.
\newblock Geometric approach to nonvariational singular elliptic equations.
\newblock {\em Arch. Ration. Mech. Anal.}, 209(3):1019--1054, 2013.

\bibitem{phys}
C.~Baiocchi, G.~Buttazzo, F.~Gastaldi and F.~Tomarelli.
\newblock General existence theorems for unilateral problems in continuum
mechanics.
\newblock {\em Arch. Rational Mech. Anal.}, 100(2):149--189, 1988.


\bibitem{MR1005611}
L. Caffarelli.
\newblock Interior a priori estimates for solutions of fully nonlinear
  equations.
\newblock {\em Ann. of Math. (2)}, 130(1):189--213, 1989.

\bibitem{ccbook}
L.~Caffarelli and X.~Cabr{\'e}.
\newblock{ Fully nonlinear elliptic equations}, volume~43 of {\em American
  Mathematical Society Colloquium Publications}.
\newblock American Mathematical Society, Providence, RI, 1995.


\bibitem{CCKS} L. Caffarelli,  M. Crandall,  M. Kocan and A. Swiech. {On viscosity solutions of fully nonlinear equations with measurable ingredients} Comm. Pure Appl. Math. 49 (1996) 365--397.


\bibitem{MR1978880}
L. Caffarelli and Q.~Huang.
\newblock Estimates in the generalized {C}ampanato-{J}ohn-{N}irenberg spaces
  for fully nonlinear elliptic equations.
\newblock {\em Duke Math. J.}, 118(1):1--17, 2003.

\bibitem{MR3266252}
H.~Dong and D.~Kim.
\newblock On the impossibility of {$W_p^2$} estimates for elliptic equations
  with piecewise constant coefficients.
\newblock {\em J. Funct. Anal.}, 267(10):3963--3974, 2014.

\bibitem{MR1237053}
L.~Escauriaza.
\newblock {$W^{2,n}$} a priori estimates for solutions to fully nonlinear
  equations.
\newblock {\em Indiana Univ. Math. J.}, 42(2):413--423, 1993.

\bibitem{Evans82}
L. Evans.
\newblock Classical solutions of fully nonlinear, convex, second order elliptic
  equations.
\newblock {\em Comm. Pure Appl. Math.}, 35(3):333--363, 1982.

\bibitem{econ}
J. Evers.
\newblock {The dynamics of concave input/output processes}.
\newblock {\it Springer}, 1979.

\bibitem{jensen}
R. Jensen.
\newblock The maximum principle for viscosity solutions of fully nonlinear second order partial differential equations.
\newblock {\em Arch. Rational Mech. Anal.}, 101:1--27, 1988.


\bibitem{JKS} R. Jensen,  M. Kocan, and A. Swiech.
\newblock Good and viscosity solutions of fully nonlinear elliptic equations.  
\newblock {\em Proc. Amer. Math. Soc.} 130 (2002), no. 2, 533--542

\bibitem{JNBMO}
F. John and L. Nirenberg.
\newblock On functions of bounded mean oscillation.
\newblock {\em Comm. Pure Appl. Math.}, 14:415--426, 1961.

\bibitem{Krylov82}
N. Krylov.
\newblock Boundedly inhomogeneous elliptic and parabolic equations.
\newblock {\em Izv. Akad. Nauk SSSR Ser. Mat.}, 46(3):487--523, 670, 1982.

\bibitem{Krylov83}
N. Krylov.
\newblock Boundedly inhomogeneous elliptic and parabolic equations in a domain.
\newblock {\em Izv. Akad. Nauk SSSR Ser. Mat.}, 47(1):75--108, 1983.

\bibitem{Krylov2012}
N. Krylov.
\newblock On the existence of smooth solutions for fully nonlinear elliptic equations with measurable "coefficients'' without convexity assumptions. 
\newblock {\em Methods Appl. Anal.} 19 (2012), no. 2, 119--146.

\bibitem{Krylov2013}
N. Krylov
\newblock On the existence of $W^2_p$ solutions for fully nonlinear elliptic equations under relaxed convexity assumptions. 
\newblock {\em  Comm. Partial Differential Equations} 38 (2013), no. 4, 687--710. 

\bibitem{KS79}
N. Krylov and M. Safonov.
\newblock An estimate for the probability of a diffusion process hitting a set
  of positive measure.
\newblock {\em Dokl. Akad. Nauk SSSR}, 245(1):18--20, 1979.

\bibitem{KS80}
N. Krylov and M. Safonov.
\newblock A property of the solutions of parabolic equations with measurable
  coefficients.
\newblock {\em Izv. Akad. Nauk SSSR Ser. Mat.}, 44(1):161--175, 239, 1980.

\bibitem{flin1}
F-H. Lin.
\newblock Second derivative {$L^p$}-estimates for elliptic equations of
  nondivergent type.
\newblock {\em Proc. Amer. Math. Soc.}, 96(3):447--451, 1986.

\bibitem{NV07}
N.~Nadirashvili and S.~Vl{\u{a}}du{\c{t}}.
\newblock Nonclassical solutions of fully nonlinear elliptic equations.
\newblock {\em Geom. Funct. Anal.}, 17(4):1283--1296, 2007.

\bibitem{NV08}
N.~Nadirashvili and S.~Vl{\u{a}}du{\c{t}}.
\newblock Singular viscosity solutions to fully nonlinear elliptic equations.
\newblock {\em J. Math. Pures Appl. (9)}, 89(2):107--113, 2008.

\bibitem{NV11}
N.~Nadirashvili and S.~Vl{\u{a}}du{\c{t}}.
\newblock Octonions and singular solutions of {H}essian elliptic equations.
\newblock {\em Geom. Funct. Anal.}, 21(2):483--498, 2011.

\bibitem{RT} G. Ricarte  and E. Teixeira. {Fully nonlinear singularly perturbed equations and asymptotic free boundaries.}  {\it J. Funct. Anal.}, vol. 261, Issue 6, 2011, 1624--1673.

\bibitem{Rockafellar}
R. Rockafellar.
\newblock Convex analysis. {\em Princeton Mathematical Series, No. 28}
\newblock {Princeton University Press, Princeton, N.J.}, 1970.

\bibitem{SilvTeix}
L.~Silvestre and E. Teixeira.
\newblock Regularity estimates for fully nonlinear equations which are
  asymptotically convex. To appear in {\it Progr. Nonlinear Differential Equations Appl.}

\bibitem{MR1404014}  J. Spruck. { Fully nonlinear elliptic equations and applications to geometry.}  {\it Proceedings of the {I}nternational {C}ongress of {M}athematicians, {Z}\"urich, 1994}, 1995, 1145--1152.


\bibitem{MR3158810}
E. Teixeira.
\newblock Universal moduli of continuity for solutions to fully nonlinear
  elliptic equations.
\newblock {\em Arch. Ration. Mech. Anal.}, 211(3):911--927, 2014.

\bibitem{MR3227432}
E. Teixeira and J-M. Urbano.
\newblock A geometric tangential approach to sharp regularity for degenerate
  evolution equations.
\newblock {\em Anal. PDE}, 7(3):733--744, 2014.

\bibitem{WY} D. Wang,  Y. Yuan 
{Singular solutions to special Lagrangian equations with subcritical phases and minimal surface systems.} 
{\em Amer. J. Math.} 135 (2013), no. 5, 1157--1177.

\bibitem{Yuan1}
Y. Yuan.
\newblock A {B}ersntein problem for special {L}agrangian equations.
\newblock {\em Invent. Math.}, 150(1):117--125, 2002.


\end{thebibliography}
%\bibliographystyle{plain}

\bigskip

%\vspace{1cm}

\noindent\textsc{Edgard A. Pimentel}\\
Department of Mathematics\\
Universidade Federal de S\~ao Carlos\\
13.560 S\~ao Carlos-SP, Brazil\\
\noindent\texttt{edgard@dm.ufscar.br}
\bigskip

\noindent\textsc{Eduardo V. Teixeira}\\
Universidade Federal do Cear\'a\\
Campus do Pici, Bloco 914,\\
60.455-760 Fortaleza-CE, Brazil.\\
\noindent\texttt{teixeira@mat.ufc.br}

\end{document}